\documentclass[a4paper]{article}
\usepackage[left=25truemm,right=25truemm,top=30truemm,bottom=30truemm]{geometry}
\usepackage{amsmath}
\usepackage{amssymb}
\usepackage{bm}
\usepackage{mathrsfs}
\usepackage{graphicx}
\usepackage{array} 
\usepackage{float} 
\usepackage[all]{xy} 
\xyoption{poly}
\usepackage{ulem} 
\usepackage{color} 

\makeatletter
\@addtoreset{figure}{section}
\@addtoreset{table}{section}
\makeatother

\usepackage{amsthm}
\theoremstyle{plain}
\newtheorem{theorem}{Theorem}[section]
\newtheorem{corollary}[theorem]{Corollary} 
\newtheorem{lemma}[theorem]{Lemma} 
\newtheorem{proposition}[theorem]{Proposition} 
\theoremstyle{definition}
\newtheorem{definition}[theorem]{Definition}
\newtheorem{notation}[theorem]{Notation}
\theoremstyle{remark}

\renewcommand{\qed}{\hfill $\Box$\par}
\newcounter{num}
\newcommand{\Rnum}[1]{\setcounter{num}{#1} \Roman{num}}

\newcommand{\bC}{\mathbb{C}}

\newcommand{\bF}{\mathbb{F}}

\newcommand{\bQ}{\mathbb{Q}}
\newcommand{\bR}{\mathbb{R}}
\newcommand{\bT}{\mathbb{T}}
\newcommand{\bZ}{\mathbb{Z}}

\newcommand{\cB}{\mathcal{B}}
\newcommand{\cC}{\mathcal{C}}
\newcommand{\cE}{\mathcal{E}}
\newcommand{\cK}{\mathcal{K}}

\newcommand{\cO}{\mathcal{O}}

\newcommand{\cV}{\mathcal{V}}
\newcommand{\cW}{\mathcal{W}}
\newcommand{\cX}{\mathcal{X}}

\newcommand{\fD}{\mathfrak{D}}
\newcommand{\fm}{\mathfrak{m}}

\newcommand{\rO}{\mathrm{O}}
\newcommand{\id}{\mathrm{id}}
\newcommand{\GL}{\mathrm{GL}}
\newcommand{\ob}{\mathrm{ob}}

\newcommand{\Br}{\operatorname{Br}}
\newcommand{\ch}{\operatorname{char}}
\newcommand{\Cor}{\operatorname{Cor}}

\newcommand{\diag}{\operatorname{diag}}
\newcommand{\disc}{\operatorname{disc}}

\newcommand{\idx}{\operatorname{idx}}
\newcommand{\Idx}{\operatorname{Idx}}
\newcommand{\im}{\operatorname{im}}
\newcommand{\inv}{\operatorname{inv}}
\newcommand{\N}{\operatorname{N}}
\newcommand{\Res}{\operatorname{Res}}

\newcommand{\sn}{\operatorname{sn}}

\newcommand{\Supp}{\operatorname{Supp}}
\newcommand{\Tr}{\operatorname{Tr}}
\newcommand{\Tw}{\operatorname{Tw}}
\newcommand{\la}{\langle}
\newcommand{\ra}{\rangle}

\title{Lattice Isometries and K3 Surface Automorphisms:\\
Salem Numbers of Degree $20$}
\author{Yuta Takada\thanks{Department of Mathematics, Graduate School of Science, 
Hokkaido University, Kita 10, Nishi 8, Kita-ku, Sapporo 060-0810 Japan; 
JSPS Research Fellow. \texttt{takada@math.sci.hokudai.ac.jp}}}
\date{February 15, 2023}

\begin{document}
\maketitle

\begin{abstract}
This article extends Bayer-Fluckiger's theorem on characteristic polynomials of
isometries on an even unimodular lattice to the case where the
isometries have determinant $-1$. As an application, we show that the logarithm of every
Salem number of degree $20$ is realized as the topological entropy of 
an automorphism of a nonprojective K3 surface.
\end{abstract}

\section{Introduction}
Let non-negative integers $r$ and $s$ be given. Which polynomial $F(X)\in\bZ[X]$ can 
occur as the characteristic polynomial of an isometry on an even unimodular lattice 
of signature $(r,s)$? 
It is known that if $(r,s)$ is the signature of an even unimodular lattice then 
$r\equiv s \bmod 8$. In the following we assume this congruence. 
For a monic polynomial $F$ with $F(0)\neq 0$, we define
$F^*(X):= F(0)^{-1} X^{\deg F} F(X^{-1})$
and say that $F$ is \textit{$*$-symmetric} if $F = F^*$.
In this case, the constant term $F(0)$ is $1$ or $-1$, so
we say that $F$ is \textit{$+1$-symmetric} or \textit{$-1$-symmetric}
according to the value $F(0)$
(see \S \ref{ss:symm_poly} for a slightly more formal definition).
Gross and McMullen \cite{GM02} raised the above question and 
gave the following necessary conditions essentially: 
If $F(X)\in\bZ[X]$ is the characteristic polynomial of an isometry on an even 
unimodular lattice of signature $(r,s)$ then 
$F$ is a $*$-symmetric polynomial of even degree;
\begin{equation}\label{eq:Sgcd}
\text{$r,s\geq m(F)$ and if $F(1)F(-1)\neq 0$ then 
$r \equiv s \equiv m(F) \bmod 2$,} \tag{Sign} 
\end{equation}
where $m(F)$ is the number of roots $\lambda$ of $F$ with $|\lambda|>1$
counted with multiplicity; and
\begin{equation}\label{eq:Sqcd}
\text{$|F(1)|, |F(-1)|$ and $(-1)^{(\deg F)/2} F(1) F(-1)$ are all squares.} \tag{Square}
\end{equation}
For an \textit{irreducible} $*$-symmetric polynomial $F$ of even degree, 
they speculated that these conditions are sufficient and showed that 
if the assumption \eqref{eq:Sqcd} is replaced
by the assumption $|F(1)|=|F(-1)| = 1$, then these are sufficient.
Afterwards, Bayer-Fluckiger and Taelman \cite{BT20} 
showed that the speculation is correct using a local-global theory.

Bayer-Fluckiger \cite{Ba20,Ba21,Ba22} proceeded to the case where
polynomials are reducible and $+1$-symmetric. 
In this case, the above conditions are not sufficient 
as pointed out in \cite{GM02}.
She showed that the condition \eqref{eq:Sqcd} is necessary and sufficient for
the existence of an even unimodular $\bZ_p$-lattice having a semisimple
isometry with characteristic polynomial $F$ for each prime $p$.
Moreover, she gave a necessary and sufficient condition for the local-global 
principle to hold.

Let $F\in\bR[X]$ be a $*$-symmetric polynomial with the condition \eqref{eq:Sgcd}. 
If $t$ is an isometry with characteristic polynomial $F$ on 
an $\bR$-inner product space $V$ of signature $(r,s)$, then $V$ decomposes as
$V = \bigoplus_{f}V(f;t)$, where 
$V(f;t) := \{ v\in V \mid f(t)^N v = 0 \text{ for some $N\geq0$} \}$ and 
$f$ ranges over the irreducible factors of $F$ in $\bR[X]$. 
The \textit{index} $\idx_t$ of $t$ is a map from the set of irreducible $*$-symmetric 
factors of $F$ to $\bZ$ defined by $ \idx_t(f) = r_f - s_f$ where $(r_f,s_f)$ is the  
signature of $V(f;t)$. 
The set of maps expressed as $\idx_t$ for some $t$ as above is denoted by 
$\Idx_{r,s}(F)$, and a map in $\Idx_{r,s}(F)$ are called an \textit{index map} 
(see \S\ref{ss:IPSoverR} for the precise definition). 
We refer to an isometry with characteristic polynomial $F$ and index 
$\iota\in\Idx_{r,s}(F)$ as an $(F,\iota)$-isometry for short. 

Now, let $F\in \bZ[X]$ be a $+1$-symmetric polynomial of even degree with the conditions 
\eqref{eq:Sgcd} and \eqref{eq:Sqcd}, and let $\iota\in\Idx_{r,s}(F)$ be an index map.
Bayer-Fluckiger introduced a group $\Omega$ and a homomorphism 
$\ob:\Omega\to \bZ/2\bZ$, 
which are determined by $F$ and $\iota$, and showed that there exists an even 
unimodular lattice of signature $(r,s)$ having a semisimple $(F, \iota)$-isometry  
if and only if the map $\ob$ is the zero map.

This article extends her work to the case where $F$ is $*$-symmetric, 
which covers the $-1$-symmetric case. 
We will redefine the group $\Omega$ and the map $\ob$ 
to work in the general case and show the following theorem.
We will call the map $\ob$ the \textit{obstruction map}. 

\begin{theorem}\label{th:main1}
Let $F\in \bZ[X]$ be a $*$-symmetric polynomial of even degree with the conditions 
\eqref{eq:Sgcd} and \eqref{eq:Sqcd}, and let $\iota\in\Idx_{r,s}(F)$ be an index map.
Then there exists an even unimodular lattice of signature $(r,s)$ having a 
semisimple $(F, \iota)$-isometry  
if and only if the obstruction map $\ob:\Omega\to \bZ/2\bZ$ vanishes. 
\end{theorem}

As an application, we show:

\begin{theorem}\label{th:mainapp2}
Let $F\in\bZ[X]$ be a $*$-symmetric polynomial of even degree $2n$ with 
the condition \eqref{eq:Sqcd}. 
If neither the multiplicity of $X-1$ nor that of $X+1$ in $F$ is $1$, 
then there exists a semisimple isometry on
an even unimodular lattice of signature $(n,n)$ with characteristic
polynomial $F$.
\end{theorem}

The question mentioned at the begining is related to the study of 
automorphisms on K3 surfaces.
For any K3 surface $\cX$, that is, a simply connected compact complex surface
with a nowhere vanishing holomorphic $2$-form, 
the middle cohomology group $H^2(\cX, \bZ)$ 
with the intersection form is an even unimodular lattice of signature $(3,19)$.
Such a lattice is called a \textit{K3 lattice}, which is uniquely determined 
up to isomorphism. Let $\Lambda$ be a K3 lattice. 
We can define an additional structure on $\Lambda$
which is called a \textit{K3 structure}.
An isometry $t$ on $\Lambda$ preserving a given K3 structure `lifts' 
to an automorphism on a K3 surface. In other words, 
there exists an automorphism $\varphi$ on a K3 surface $\cX$ such that
the induced homomorphism $\varphi^*:H^2(\cX,\bZ)\to H^2(\cX,\bZ)$
can be identified with the isometry $t$.
This is a consequence of the Torelli theorem and surjectivity of the period
mapping.

There are many studies of dynamics on K3 surfaces using the lifting, 
see \cite{Mc02,Og10,Mc11,Mc16,IT22,IT23} for instance.
This article focuses on the topological entropy. 
It is known that the topological entropy of a K3 surface automorphism $\varphi$
coincides with the logarithm of the spectral radius $\lambda(\varphi^*)$ of 
$\varphi^*: H^2(\cX, \bC) \to H^2(\cX, \bC)$,  
and that $\lambda(\varphi^*)$ is a Salem number (see \S\ref{ss:symm_poly}) unless
$\lambda(\varphi^*) = 1$ (or the entropy equals $0$).
Let us say that a Salem number $\lambda$ is 
\textit{projectively} (resp.~\textit{nonprojectively}) \textit{realizable}
if there exists an automorphism on a projective (resp. nonprojective)
K3 surface with entropy $\log\lambda$.
We remark that the degree of such a Salem number is an even integer
between $2$ and $20$ (resp. $4$ and $22$).  

McMullen proved that the Lehmer number $\approx 1.17628$, which is 
the smallest known Salem number, is nonprojectively realizable in \cite{Mc11},
and projectively realizable in \cite{Mc16}. 
Moreover, Bayer-Fluckiger and Taelman showed in \cite{BT20} that a Salem number of 
degree $22$ is nonprojectively realizable if and only if its minimal polynomial
satisfies the condition \eqref{eq:Sqcd}, and Bayer-Fluckiger proved in \cite{Ba21, Ba22} 
that all Salem numbers of degree $4,6,8,12,14$ or $16$ are nonprojectively realizable, 
using her theorems on characteristic polynomials of isometries 
on an even unimodular lattice. Along this line, we show:

\begin{theorem}\label{th:mainK3}
All Salem numbers of degree $20$ are nonprojectively realizable.
\end{theorem}

To prove this theorem, we will use Theorem \ref{th:main1} to show that the polynomial 
of the form $(X-1)(X+1)S(X)$ for each Salem polynomial $S(X)$ of degree $20$ is 
realizable as the characteristic polynomial of an isometry on a K3 lattice. 
See Theorem \ref{th:mainapp} for a more general consequence of 
Theorem \ref{th:main1}.
We remark that the polynomial $(X-1)(X+1)S(X)$ is $-1$-symmetric, 
so Theorem \ref{th:mainK3} is a benefit of extending Bayer-Fluckiger's theory. 
The cases of degree $10$ and degree $18$ are still open, 
but there are some criteria for realizablity and 
it is known that the smallest Salem number of degree $18$ is not nonprojectively 
realizable, see \cite{Ba22}. 
Note that the lattice theoretic approach brings more subtle problems
in the projective case, see \cite{Mc16}.
We refer to \cite{Br20} for other results on entropy spectra.
\smallskip

The organization of this article is as follows.
We review fundamental facts on inner product spaces in \S \ref{sec:IPandI}
and on lattice theory in \S\ref{sec:ULandEWG}. These sections contain no 
new results. 
In \S \ref{sec:LocalTheory}, we reproduce local theory on 
even unimodular lattices and characteristic polynomials of isometries, 
which is given by Bayer-Fluckiger in \cite{Ba21}.
Moreover, we extend her theory to isometries of determinant $-1$
on an even unimodular $\bZ_2$-lattice.
For this, Theorem \ref{th:detailedQ2} is crucial.
In \S\ref{sec:LGP}, we give a necessary and sufficient condition for a pair 
$(F, \iota)$ of a polynomial $F$ and an index map $\iota$ to be realized on an 
even unimodular lattice.
In \S\ref{sec:LGO}, we reformulate the obstruction group and map, 
and establish Theorem \ref{th:main1}, extending her theory.
Theorem \ref{th:mainapp2} is also proved in this section.
In the last section \S\ref{sec:Aut_of_K3}, 
we deal with entropy problems for K3 surface automorphisms, and 
prove Theorem \ref{th:mainK3}.

\section{Inner products and isometries}\label{sec:IPandI}
Let $K$ be a field of characteristic $\neq 2$.  
An \textit{inner product space} $(V, b)$ over $K$ is a pair of
a finite dimensional $K$-vector space $V$ and an \textit{inner product} 
$b:V\times V \to K$, that is, a nondegenerate symmetric bilinear form.
We may write either of $V$ or $b$ for $(V, b)$. 
Two inner product spaces $(V, b)$ and $(V', b')$ are \textit{isomorphic} 
if there exists a linear isomorphism $\varphi:V\to V'$ satisfying 
$b(x,y) = b'(\varphi(x), \varphi(y))$ for all $x,y\in V$.
The group of isometries on $V$ is denoted by $\rO(V)$.

\subsection{Symmetric polynomials}\label{ss:symm_poly}
For a polynomial $F(X) \in K[X]$, define 
$F^\vee(X) :=\allowbreak X^{\deg F} F(X^{-1}) \in K[X]$.
A polynomial $F$ is \textit{$\epsilon$-symmetric} if $F(X) = \epsilon F^\vee(X)$ 
where $\epsilon = \pm 1$. Such a polynomial occurs naturally as the 
characteristic polynomial of an isometry (see \S\ref{ss:isometries}).
For a monic polynomial $F(X) \in K[X]$ with $F(0)\neq 0$, define
\[ F^*(X):= F(0)^{-1} X^{\deg F} F(X^{-1}) = F(0)^{-1} F^\vee(X), \]
and we say that $F$ is \textit{$*$-symmetric} if $F = F^*$.
A monic polynomial $F(X) \in K[X]$ with $F(0)\neq 0$ 
is $*$-symmetric if and only if $F$ is $+1$-symmetric or $-1$-symmetric.
Following \cite{Ba15} we say that a $*$-symmetric polynomial $F$ is of
\begin{itemize}
\item \textit{type $0$} if $F$ is a product of powers of $(X-1)$ and of $(X+1)$;
\item \textit{type $1$} if $F$ is a product of powers of $+1$-symmetric 
irreducible monic polynomials of even degrees;
\item \textit{type $2$} if $F$ is a product of polynomials of the form $GG^*$, where
$G$ is monic, irreducible and $G^* \neq G$.
\end{itemize}
Every $*$-symmetric polynomial $F$ admits a unique factorization
$F = F_0F_1F_2$, where $F_i$ is of type $i$ for $i=0,1,2$ 
(see \cite[Proposition 1.3]{Ba15}). This is true even if $\ch K = 2$. 
We refer to the factor $F_i$ as the \textit{type $i$ component} of $F$. 
For a $*$-symmetric polynomial $F$, we have for example:
\begin{itemize}
\item If $F$ has no type $0$ component, that is, $F(1)F(-1)\neq 0$, 
then $F$ is of even degree and $+1$-symmetric.
\item $F$ is $-1$-symmetric if and only if $F$ has the factor $X-1$ with 
odd multiplicity.
\end{itemize}
We remark that if $K = \bQ$ and $F\in \bZ[X]$ then the factorization  
$F = F_0 F_1 F_2$ in $\bQ[X]$ is the same as that in $\bZ[X]$, 
thanks to Gauss's lemma. 

Typical examples of $*$-symmetric polynomials
which appear in this article are 
cyclotomic polynomials and Salem polynomials.
A cyclotomic polynomial is the minimal polynomial of a root of unity, 
and a \textit{Salem polynomial} is the minimal polynomial of a 
\textit{Salem number}, that is, a real algebraic unit $\lambda > 1$ 
whose conjugates other than $\lambda^{\pm 1}$ lie on the unit circle in $\bC$.
We allow Salem numbers of degree $2$ following \cite{Mc02}, 
see \cite[p.~26]{Sa63} for Salem's definition.

\subsection{Inner products}
Let $(V, b)$ be an inner product space over $K$.
If $e_1, \ldots, e_d$ is a basis of $V$,
the matrix $(b(e_i, e_j))_{ij}$ is called the \textit{Gram matrix} of 
$e_1, \ldots, e_d$. A Gram matrix of $V$ is the Gram matrix of some basis.
The \textit{determinant} of $(V, b)$ is the element of $K^\times/K^{\times 2}$
represented by $\det G$, where $G$ is any Gram matrix of $V$.
Here $K^\times$ is the group of invertible elements and
$K^{\times 2} := \{ a^2 \mid a\in K^\times \}$.
The determinant of $(V, b)$ is denoted by $\det b$, and we define
the \textit{discriminant} $\disc b$ of $(V,b)$ to be 
$(-1)^{d(d-1)/2}\det b \in K^\times/K^{\times 2}$, where $d$ is the 
dimension of $V$.

It is well known that every inner product space has an orthogonal basis, that is, 
a basis of which the Gram matrix is diagonal. Based on this fact, 
the Hasse-Witt invariant is defined as follows. 
First of all, we denote by $\Br(K)$ the Brauer group of $K$, and by
$(a,b)\in \Br(K)$ the Brauer class of the quaternion 
algebra defined by $a, b\in K^\times$ 
(see e.g. \cite[Chapter 8]{Sc85} for Brauer groups).
Let $e_1, \ldots, e_d$ be an orthogonal basis of $(V,b)$ 
with the Gram matrix $\diag(a_1, \ldots ,a_d)$. Then we define
\[ \epsilon(b) := \sum_{i<j} (a_i, a_j) \in \Br(K), \]
where we regard $\Br(K)$ as an additive group.
The element $\epsilon(b)$ does not depend on the choice of the orthogonal
basis. We call $\epsilon(b)$ the \textit{Hasse-Witt invariant} of $b$.  
The Hasse-Witt invariant of any $1$-dimensional inner product space is defined to be $0$.

If $(V,b)$ and $(V',b')$ are two inner product spaces, we have
$\epsilon(b\oplus b') = \epsilon(b) + \epsilon(b') +\allowbreak 
(\det b, \det b')$.
This formula implies the following lemma. 

\begin{lemma}\label{lem:HWsumFormula}
Let $V_1, \ldots, V_k$ be $K$-vector spaces, and let 
$b_j$ and $b_j'$ be inner products on $V_j$ for $j = 1, \ldots, k$.
If $\det b_j = \det b_j'$ for all $j$, then we have
\[ \epsilon\left( \bigoplus_{j=1}^k b_j \right)
- \sum_{j=1}^k \epsilon(b_j) 
= \epsilon\left( \bigoplus_{j=1}^k b_j' \right)
- \sum_{j=1}^k \epsilon(b_j'). \]
\end{lemma}

If $K$ is a local field, Hasse-Witt invariants will often be considered to 
take values in $\{0,1\} = \bZ/2\bZ$ because the image of $\epsilon$ is a subgroup 
of order $2$ in $\Br(K)$.
Suppose that $K$ is the field of rational numbers $\bQ$, and let $v$ be a place.
In this case Hasse-Witt invariants over the $v$-adic field $\bQ_v$ 
or over $\bQ_\infty = \bR$ will be denoted 
by $\epsilon_v:\{\text{inner products over $\bQ_v$}\} \to \bZ/2\bZ$.
Here $\infty$ denotes the infinite place of $\bQ$. 
Furthermore, for an inner product $b$ over $\bQ$, we will write 
$\epsilon_v(b)$ for $\epsilon_v(b\otimes \bQ_v)$.

\subsection{Isometries}\label{ss:isometries}
Let $(V, b)$ be an inner product space and $t\in \rO(V)$ an isometry
with characteristic polynomial $F$. 
As mentioned earlier, the polynomial $F$ is $*$-symmetric 
(see \cite[Proposition A.1]{GM02} and its proof).
Let us define 
\[ V(f;t) := \{ v\in V \mid f(t)^Nv = 0 \quad\text{for some $N\geq 0$} \} \]
for $f\in K[X]$. Then we have an orthogonal direct sum
decomposition
\[ V = \bigoplus_{f\in I_0\cup I_1} V(f;t) \oplus 
\bigoplus_{\{f,f^*\}\subset I_2} [V(f;t)\oplus V(f^*;t)],  \]
where $I_i$ denotes the set of irreducible factors of the type $i$ component of $F$ 
for $i = 0,1,2$.
Furthermore, for each $f\in I_2$, the component $V(f;t)\oplus V(f^*;t)$ is 
\textit{split}, i.e., 
$V(f;t)\oplus V(f^*;t)$ has a Gram matrix of the form
\[\begin{pmatrix}
  0 & I \\ I & 0
\end{pmatrix}, \]
where $I$ denotes the identity matrix, see \cite[\S 3]{Mi69}.

There is a relation between the determinant $\det b$ and the polynomial $F$.

\begin{lemma}\label{lem:det_iso_formula}
If $F(1)F(-1) \neq 0$ then $\det b = F(1)F(-1)$ in $K^\times/K^{\times 2}$. 
In particular $\det b|_{V(f;t)\oplus V(f^*;t)} = (-1)^{\deg f}$ in 
$K^\times/K^{\times 2}$ for $f\in I_2$.
\end{lemma}
\begin{proof}
See \cite[Corollary 5.2]{Ba15}.
\end{proof}

\subsection{Algebras with involution and hermitian forms}
In this article a $K$-algebra means an associative unital $K$-algebra.
We always assume that $K$-algebras are finite dimensional over $K$.
Moreover, $K$-algebras will often be equipped with an involution.
If $A = (A,\sigma)$ is a $K$-algebra with an involution $\sigma: A\to A$,
define the fixed subalgebra $A^\sigma$ to be
$\{ x \in A \mid  \sigma (x) = x \}$.
Two $K$-algebras $A = (A, \sigma)$ and $A' = (A', \sigma')$, 
where $\sigma$ and $\sigma'$ are involutions, 
are \textit{isomorphic} if 
there exists an isomorphism $\varphi:A\to A'$ as $K$-algebra satisfying 
$\varphi\circ \sigma = \sigma' \circ \varphi$.

Let $A = (A, \sigma)$ be a $K$-algebra with involution and $M$ an $A$-module. 
A map $h:M\times M \to A$ is called a \textit{hermitian form}
on $M$ if $h$ is $A$-linear in the first variable and 
$h(x,y) = \sigma h(y,x)$ for all $x, y\in M$.
If $M$ is free over $A$, a \textit{Gram matrix} and the \textit{determinant} of 
$h$ are defined in the same way as for an inner product. 
The determinants of hermitian forms take values in 
\[ \Tw(A,\sigma) := (A^\sigma)^\times / \N_{A/A^\sigma} (A^\times), \]
which will be called the \textit{twisting group} of $(A, \sigma)$,
where $\N_{A/A^\sigma}:A^\times \to (A^\sigma)^\times$ is the norm map.

Let $f\in K[X]$ be a $*$-symmetric polynomial of even degree
and suppose that $f$ is irreducible or of the form $gg*$ for some irreducible
monic polynomial $g$ satisfying $g \neq g^*$.
The $K$-algebra $E:=K[X]/(f)$ has an involution $\sigma$ defined
by $\alpha\mapsto \alpha^{-1}$, where $\alpha = X + (f)\in E$ is the image of $X$.
Let $M$ be a direct product of finite copies of $E$.
We can regard $M$ as an $E$-module and also as a $K$-vector space. 
Hermitian forms on $M$ over $E$ and inner products on $M$ over $K$
which make $\alpha:M \to M$ (multiplication by $\alpha$) an isometry are 
related as follows.

\begin{lemma}\label{lem:tr_hermitian}
Let $f$ and $M$ be as above, and assume that $f$ is separable. 
If $b$ is an inner product on the $K$-vector space $M$ 
such that $\alpha$ becomes an isometry, then there exists one and only one 
hermitian form $h$ on $M$ over $E$ such that $b = \Tr_{E/K}\circ h$.
Here $\Tr_{E/K}$ is the trace map.
\end{lemma}
\begin{proof}
This statement for an irreducible $f$ is in \cite[Lemma 1.1]{Mi69} and 
the same proof is valid for $f = gg*$. Note that separablility of $f$ 
guarantees that the trace map is non-zero.
\end{proof}

\subsection{Inner product spaces over $\bR$}\label{ss:IPSoverR}
In this subsection we introduce an invariant of an isometry on a real inner product 
space called the index (cf. \cite[Section 6]{Ba21}). 
Let $(V,b)$ be an inner product space over $\bR$ of signature $(r,s)$.
The \textit{index of $(V, b)$} is $r - s$ and we denote this by 
$\idx V$ or $\idx b$.
Notice that $r = (\dim V + \idx V)/2$ and $s = (\dim V - \idx V)/2$. 
Let $t\in \rO(V)$ be an isometry with characteristic polynomial 
$F \in \bR[X]$.
The symbol $I_i(\bR)$ denotes the set of irreducible factors (in $\bR[X]$) of the 
type $i$ component of $F$ for $i = 0,1,2$.
Note that any $f\in I_1(\bR)$ is expressed as 
$f(X) = X^2 - (\delta + \delta^{-1})X + 1$ for some $\delta \in \bT\setminus\{\pm 1\}$, 
where $\bT := \{ \delta\in\bC \mid |\delta| = 1 \}$. 
The \textit{index of $t$}, denoted by $\idx_t$, is a map from 
$I_0(\bR)\cup I_1(\bR)$ to $\bZ$ defined by 
$\idx_t(f) = \idx V(f;t)$.
The reason for not including $I_2(\bR)$ in the domain of definition is in 
\textup{(ii)} of the following proposition.

\begin{proposition}\label{prop:Sgcd}
Let $(r_f, s_f)$ be the signature of $V(f;t)$ for $f\in I_1(\bR)$, and
that of $V(f;t)\oplus V(f^*;t)$ for $f\in I_2(\bR)$.
Then we have
\begin{enumerate}
\item $r_f \equiv s_f \equiv 0 \bmod 2$ for each $f\in I_1(\bR)$; and
\item $r_f = s_f$ (i.e. $\idx (V(f;t)\oplus V(f^*;t)) = 0$) for each $f\in I_2(\bR)$.
\end{enumerate}
In particular, we have
\begin{equation*}
\text{$r,s\geq m(F)$ and if $F(1)F(-1)\neq 0$ then 
$r \equiv s \equiv m(F)\bmod 2$} \tag{Sign}
\end{equation*}
where $m(F)$ is the number of roots $\lambda$ of $F$ with $|\lambda|>1$
counted with multiplicity.
\end{proposition}
\begin{proof}
\textup{(i)} is a consequence of Lemma \ref{lem:tr_hermitian}, see the proof of 
\cite[Proposition 8,1 (a)]{Ba15}.
\textup{(ii)} follows from the fact that $V(f;t)\oplus V(f^*;t)$ is split for 
$f\in I_2(\bR)$.
\end{proof}

Conversely, let $F\in \bR[X]$ be a $*$-symmetric polynomial of degree $d$, and 
$r, s$ non-negative integers with $r + s = d$, and assume that
the condition \eqref{eq:Sgcd} holds.
We denote by $n_+,n_-$ and $n_f$ the multiplicities of $X-1, X+1$ and $f\in I_1(\bR)$ 
in $F$ respectively. Let us define $\Idx_{r,s}(F)$ to be the set of maps 
$\iota: I_0(\bR)\cup I_1(\bR)\to \bZ$ such that
\begin{equation}\label{eq:idxmap0}
\text{$\iota(X\mp 1)\equiv n_\pm$ mod $2$ and 
$-n_\pm \leq \iota(X\mp 1) \leq n_\pm$;}
\end{equation}
\begin{equation}\label{eq:idxmap1}
\begin{split}
&\text{$\iota(f)$ is even, $-2n_f \leq \iota(f) \leq 2n_f$, and} \\
&\text{$(2n_f + \iota(f))/2 \equiv (2n_f - \iota(f))/2 \equiv 0$ mod $2$
for $f\in I_1(\bR)$; and}
\end{split}
\end{equation}
\begin{equation}\label{eq:idxmapsum}
\sum_{f\in I_0(\bR)\cup I_1(\bR)} \iota(f) = r-s.
\end{equation}
Proposition \ref{prop:Sgcd} implies that $\idx_t$ belongs to $\Idx_{r,s}(F)$ 
for any isometry $t$ with characteristic polynomial $F$ on 
a $d$-dimensional inner product space over $\bR$ of signature $(r,s)$.
We call a map in $\Idx_{r,s}(F)$ an \textit{index map}.

\begin{proposition}\label{prop:realize_idx}
Let $F$ and $r,s$ be as above with the condition \eqref{eq:Sgcd}. 
For any index map $\iota \in \Idx_{r,s}(F)$, 
there exists an inner product space having a semisimple isometry
with characteristic polynomial $F$ and index $\iota$. 
\end{proposition}
\begin{proof}
See the proof of \cite[Proposition 8,1 (b)]{Ba15} and \cite[Proposition 7.1]{Ba21}.
\end{proof}

\section{Unimodular lattices and equivariant Witt groups}\label{sec:ULandEWG}
In this section, we review some terms and known results of lattice theory.

\subsection{Lattices}
Let $R$ be a Dedekind domain and $K$ its field of fractions. 
A \textit{lattice} over $R$ (or \textit{$R$-lattice}) is a pair $(\Lambda, b)$ of 
a finitely generated free $R$-module $\Lambda$ and an inner product 
$b:\Lambda\times \Lambda\to K$. 
Let $\Lambda = (\Lambda, b)$ be a lattice over $R$. The lattice 
$\Lambda^\vee 
:= \{ y\in \Lambda\otimes_R K \mid b(x,y)\in R \text{ for all $x \in \Lambda$} \}$
is called the \textit{dual lattice} of $\Lambda$. 
The lattice $(\Lambda, b)$ is said to be \textit{$R$-valued} if $b$ takes values 
in $R$. It is obvious that $\Lambda$ is 
$R$-valued if and only if $\Lambda \subset \Lambda^\vee$.
We say that $\Lambda$ is \textit{unimodular} if $\Lambda^\vee = \Lambda$, and
\textit{even} if $b(x,x)\in 2R$ for all $x\in \Lambda$ and \textit{odd} otherwise.
Note that if $2$ is a unit of $R$ then any lattice is even.

Let $\bZ_2$ denote the ring of $2$-adic integers.
We will use the following facts, see e.g. \cite[\S 106 A]{OM73}.

\begin{proposition}\label{prop:Z2-EUL}
Any even unimodular $\bZ_2$-lattice is of even rank $2n$ and isomorphic to
$ U^{\oplus n}$ or $U^{\oplus n-1}\oplus V$,
where $U$ and $V$ are $\bZ_2$-lattices of rank $2$ which have Gram matrices 
\[ \begin{pmatrix}
  0 & 1 \\ 1 & 0
\end{pmatrix}
\quad \text{ and } \quad
\begin{pmatrix}
  2 & 1 \\ 1 & 2
\end{pmatrix} \]
respectively. In particular, its discriminant equals $1$ or $-3$
in $\bQ_2^\times / \bQ_2^{\times 2}$. 
\end{proposition}

\begin{proposition}
Let $r$ and $s$ be non-negative integers. There exists an even unimodular
$\bZ$-lattice of signature $(r,s)$ if and only if $r\equiv s \bmod 8$.
\end{proposition}

\subsection{Equivariant Witt groups}
In this subsection we review equivariant Witt groups, see \cite{BT20} for more details.
Let $K$ be a field and $G$ a group. 
The group ring $K[G]$ has a $K$-linear involution $\sigma$ induced by $g\mapsto g^{-1}$
for all $g\in G$. 
If $G$ is commutative then $\sigma$ is a $K$-algebra involution.
We will mostly consider the case where $G$ is the infinite cyclic group. 
A \textit{$K[G]$-bilinear form} or just \textit{$K[G]$-form} is a pair $(V,b)$ of
a $K[G]$-module $V$ which is finite dimensional over $K$ and an inner product $b$
on $V$ over $K$ satisfying 
\[ b(ax,y) = b(x, \sigma(a)y) \quad \text{for all $x,y \in V$ and $a\in K[G]$}. \]
If $(V,b)$ is a $K[G]$-form then each $g\in G$ acts on $V$ as an isometry.
Hence, a $K[G]$-form $(V, b)$ can be regarded as a triple $(V, b, \rho)$ 
consisting of a $K$-vector space $V$, an inner product $b$ on $V$, and 
an `orthogonal' representation $\rho: G\to \rO(V)$ of $G$. 
In this case, the $K[G]$-form $(V, b)$ is also denoted by $(V, b, \rho)$. 

The \textit{equivariant Witt group} $W_G(K)$ is defined as an analog of the usual 
Witt group $W(K)$ of $K$ (see \cite[Definition 3.3]{BT20}). 
The Witt class represented by a $K[G]$-form $(V, b)$ (or $(V, b, \rho)$) is denoted 
by $[V,b]$ (or $[V, b, \rho]$).
For an irreducible representation $\rho: G\to \GL(V)$
on a finite dimensional vector space $V$, we define
$W_G(K;\rho)$ as the subgroup of $W_G(K)$ generated by classes which can be
written as $[V,b,\rho]$ for some inner product $b$ on $V$.
Then we have a decomposition
\[ W_G(K) = \bigoplus_{\rho} W_G(K;\rho) \]
where $\rho$ ranges over the isomorphism classes of finite dimensional 
irreducible representations of $G$.

The theory of (equivariant) Witt groups can be used to discuss the existence of a 
unimodular lattice in an inner product space over a discrete valuation field.
Assume that $K$ is a discrete valuation field with valuation $v$, valuation 
ring $\cO$, and residue class field $k$.
A $K[G]$-form $V$ is \textit{bounded} if $V$ contains a 
\textit{$G$-stable lattice}, that is, a lattice $\Lambda$ over $\cO$ in $V$ 
such that $K\Lambda = V$ and $g\Lambda = \Lambda$ for all $g\in G$.
We denote by $W_G^\mathrm{b}(K)$ the subgroup of $W_G(K)$ generated by Witt 
classes of bounded $K[G]$-forms.

Let $(V, b)$ be a bounded $K[G]$-form. Then there exists an 
\textit{almost unimodular} lattice $\Lambda$ in $V$, that is, a lattice in $V$
satisfying $\pi \Lambda^\vee \subset \Lambda \subset \Lambda^\vee$, 
where $\pi$ is a uniformizer of $K$.
For an almost unimodular lattice $\Lambda$ we can define a $k[G]$-form 
$(\Lambda^\vee/\Lambda, \bar{b})$ by
\[ \bar{b} : \Lambda^\vee/\Lambda \times \Lambda^\vee/\Lambda
\stackrel{b}{\to} \pi^{-1}\cO/\cO
\stackrel{\times \pi}{\longrightarrow} \cO/\pi \cO = k. \]
Let $\partial:W_G^\mathrm{b}(K) \to W_G(k)$ denote the map sending 
the class of a bounded $K[G]$-form $(V,b)$ to the class of 
the $k[\Gamma]$-from $(\Lambda^\vee/\Lambda, \bar{b})$
for some almost unimodular lattice $\Lambda$ in $V$. 
We remark that $\dim \partial[V, b] \equiv v(\det b) \bmod 2$.
See \cite[Theorem B]{BT20} for the following theorem.

\begin{theorem}\label{th:ULvanish}
The map $\partial:W_G^\mathrm{b}(K) \to W_G(k)$ is a well-defined homomorphism.
A bounded $K[G]$-form $(V,b)$ contains a $G$-stable unimodular lattice if and 
only if $\partial[V,b] = 0$.
\end{theorem}

Here we prepare some notations related to the infinite cyclic group $\Gamma$ and 
$K[\Gamma]$-forms, and will use them throughout this article.

\begin{notation}
The symbol $\Gamma$ denotes the infinite cyclic group. 
Let $\gamma$ be a fixed generator of $\Gamma$, and let $\rho$ be 
the orthogonal representation on an inner product space $(V, b)$ defined by 
$\gamma \mapsto t\in \rO(V)$ for a given $t$.
Then we write $(V, b, t)$ for $(V, b, \rho)$.
Furthermore, if the representation $\rho$ is a direct sum of copies of 
an irreducible representation $\chi$, we write
$W_\Gamma(K;t)$ for $\allowbreak W_\Gamma(K;\chi)$.
\end{notation}

For any field $K$, the subgroups $W_\Gamma(K;1)$ and $W_\Gamma(K;-1)$
of $W_\Gamma(K)$ are isomorphic to the usual Witt group $W(K)$.
We recall the structure of Witt groups of finite fields. 
See e.g. \cite[\S 2.3]{Sc85} for the following proposition.

\begin{proposition}\label{prop:WGofFF}
Let $k$ be a finite field.
\begin{enumerate}
\item If $\ch k = 2$ then sending 
$\omega\in W(k)$ to $\dim \omega \bmod 2\in \bZ/2\bZ$
gives an isomorphism $W(k)\cong \bZ/2\bZ$.
\item If $\ch k \neq 2$ then 
\[ W(k) \cong \begin{cases}
  \bZ/2\bZ \times \bZ/2\bZ & \text{if $\ch k \equiv 1 \bmod 4$}\\ 
  \bZ/4\bZ & \text{if $\ch k \equiv 3 \bmod 4$},
\end{cases}\]
and $\omega\in W(k)$ is the trivial class 
if and only if $\dim \omega \equiv 0\bmod 2$
and $\disc \omega = 1$ in $k^\times/k^{\times 2}$.
\end{enumerate}
\end{proposition}

\section{Local theory}\label{sec:LocalTheory}
Let $K$ be a non-archimedean local field of characteristic $0$. 
We denote by $v_K, \cO_K$, and $\fm_K$ the normalized valuation, valuation 
ring, and maximal ideal respectively, and fix a uniformizer $\pi_K$.
Similar notation will be used for an extension of $K$.
The residue class field of $K$ is denoted by $k$.

We can describe a necessary and sufficient condition for a unimodular 
$\cO_K$-lattice having a semisimple isometry with characteristic polynomial 
$F$ to exist, in terms of the valuations of $F(1)$ and $F(-1)$.
We adopt the convention that $v_K(0)\equiv 0 \bmod 2$.

\begin{theorem}[{\cite[Theorem 4.1]{Ba21}}]\label{th:localcdforUL}
Let $F(X) \in \cO_K[X]$ be a $*$-symmetric polynomial of even degree.
The following are equivalent:
\begin{enumerate}
\item There exists a unimodular $\cO_K$-lattice having a semisimple isometry 
with characteristic polynomial $F$.
\item $v_K(F(1)) \equiv v_K(F(-1)) \equiv 0 \bmod 2$ if $\ch k \neq 2$ and
$v_K(F(1)F(-1)) \equiv 0 \bmod 2$ if $\ch k = 2$.
\end{enumerate}
\end{theorem}

Moreover, if $K$ is the $2$-adic field $\bQ_2$ we have:

\begin{theorem}\label{th:localcdforEUL}
Let $F(X) \in \bZ_2[X]$ be a $*$-symmetric polynomial of even degree $2n$.
There exists an even unimodular $\bZ_2$-lattice of discriminant $1$ having a 
semisimple isometry with characteristic polynomial $F$
if and only if $F$ satisfies the following conditions:
\begin{itemize}
\item[\textup{(a)}] $v_2(F(1)) \equiv v_2(F(-1)) \equiv 0 \bmod 2$; and
\item[\textup{(b)}] If $F(1)F(-1) \neq 0$ then 
$(-1)^n F(1)F(-1) = 1 \in \bQ_2^\times / \bQ_2^{\times 2}$.
\end{itemize}
\end{theorem}

This is an extension of Theorems 5.1 and 5.2 of \cite{Ba21} to 
the case where $F$ is $*$-symmetric, which covers the $-1$-symmetric case. 
The proof of this theorem will be given
in the last part of this section.
Theorems \ref{th:localcdforUL} and \ref{th:localcdforEUL} yield 
an important corollary (cf. \cite[Proposition 8.1]{Ba21}).

\begin{corollary}
Let $F\in \bZ[X]$ be a $*$-symmetric polynomial of even degree $2n$.
There exists an even unimodular $\bZ_p$-lattice having a semisimple isometry 
with characteristic polynomial $F$ for each prime $p$ if and only if
$F$ satisfies the following condition:
\begin{equation*}
\text{$|F(1)|, |F(-1)|$ and $(-1)^n F(1) F(-1)$ are all squares.} \tag{Square}
\end{equation*}
\end{corollary}
\begin{proof}
Assume that there exists an even unimodular $\bZ_p$-lattice 
$\Lambda_p$ having 
a semisimple isometry with characteristic polynomial $F$ for each prime $p$.
Then $|F(1)|$ and $|F(-1)|$ are squares because
$v_p(F(\pm 1)) \equiv 0 \bmod 2$ for any prime $p$ by Theorems 
\ref{th:localcdforUL} and \ref{th:localcdforEUL}, 
where $v_p$ is the $p$-adic valuation. 
If $F(1)F(-1) = 0$ then $F$ satisfies \eqref{eq:Sqcd}.
Let $F(1)F(-1) \neq 0$. Then 
$|(-1)^n F(1)F(-1)| =\allowbreak |F(1)||F(-1)|$ 
is a square in $\bQ^\times$ 
and thus $(-1)^n F(1)F(-1) = 1$ or $-1$ in $\bQ_2^\times/\bQ_2^{\times 2}$.
On the other hand, we have $(-1)^n F(1)F(-1) = \disc \Lambda_2 \in \{1, -3\}$
by Lemma \ref{lem:det_iso_formula} and Proposition \ref{prop:Z2-EUL},  
and hence $(-1)^n F(1)F(-1) = 1$ in $\bQ_2^\times/\bQ_2^{\times 2}$.
These mean that $(-1)^n F(1)F(-1) = 1$ in $\bQ^\times/\bQ^{\times 2}$.
The converse is straightforward from Theorems 
\ref{th:localcdforUL} and \ref{th:localcdforEUL}.
\end{proof}

In \S \ref{ss:LT_EUL}, for a given $*$-symmetric polynomial $F\in \cO_K[X]$, 
we define a $K$-vector space $M$ having a semisimple linear map $\alpha :M\to M$
with characteristic polynomial $F$. 
Then we consider when $M$ admits an inner product $b$ such that 
$\alpha$ becomes an isometry on $(M,b)$ and $\partial[M,b,\alpha]$ 
vanishes in $W_\Gamma(k)$ (recall Theorem \ref{th:ULvanish}).
For this purpose, in \S \ref{ss:imageofmaps}, we make some preparations in a 
slightly more general situation.

\subsection{The map $\partial_{M,\alpha}:\Tw(E,\sigma)\to W_\Gamma(k)$}\label{ss:imageofmaps}

Let $E$ be a $K$-algebra with a nontrivial involution $\sigma$, and
assume that the fixed algebra $E^\sigma$ is a field and one of the following
holds:
\begin{itemize}
\item[(sp)] $E \cong E^\sigma\times E^\sigma$ where the involution on the 
right-hand side is the transposition of the first and second components.
\item[(ur)] $E$ is an unramified extension field of $E^\sigma$.
\item[(rm)] $E$ is a ramified extension field of $E^\sigma$.
\end{itemize}
One can see that $\Tw(E,\sigma)$ is a trivial group if $E$ is of type (sp), 
and that $\Tw(E,\sigma)\cong \bZ/2\bZ$ if $E$ is of type (ur) or (rm).

Let $M$ be a finitely generated free $E$-module. 
A hermitian form on $M$ is uniquely determined
by its dimension and determinant. More precisely, the following holds.

\begin{proposition}\label{prop:1to1_herm}
Sending a hermitian form $h$ on $M$ to its determinant $\det h$ 
gives rise to a one-to-one correspondence between the isomorphism classes of
hermitian forms on $M$ and the elements of the twisting group $\Tw(E, \sigma)$. 
\end{proposition}
\begin{proof}
Surjectivity is obvious and see \cite[Example 10.1.6 (ii)]{Sc85}
for injectivity.  
\end{proof}

If $b : M\times M \to K$ is an inner product which can be written as 
$b = \Tr_{E/K}\circ h$ for some hermitian form $h$ over $E$, 
Proposition \ref{prop:1to1_herm} implies that the isomorphism class of $b$ 
is determined by $\det h$.

\begin{notation}
Let $\lambda \in \Tw(E, \sigma)$. 
The symbol $b[\lambda]$ denotes an inner product of the form $\Tr_{E/K}\circ h$ on $M$, 
where $h$ is a hermitian form with $\det h = \lambda$. 
The isomorphism class of $b[\lambda]$ is uniquely determined by $\lambda$. 
\end{notation}

We now fix $\alpha\in \cO_E^\times$ satisfying $\alpha\sigma(\alpha) = 1$. 
Then $\alpha : M \to M$ is an isometry on  $(M,b[\lambda])$ for any 
$\lambda\in \Tw(E, \sigma)$, and the triple $(M,b[\lambda], \alpha)$ becomes 
a $K[\Gamma]$-form. Hence, we obtain the map
\[ \partial_{M, \alpha}:\Tw(E, \sigma) \to W_\Gamma(k), \,
\lambda\mapsto \partial[M, b[\lambda], \alpha]
\]
where $\partial$ is the map mentioned in Theorem \ref{th:ULvanish}. 
The purpose of this subsection is to describe the image of this map.

We begin with the case where $M$ is of rank one, that is, $M=E$.
In this case, for each $\lambda\in (E^\sigma)^\times$, the inner product 
$b_\lambda :E\times E \to K$ defined by
\[ b_\lambda(x,y) = \Tr_{E/K}(\lambda x \sigma(y)) \quad\text{for $x,y\in E$} \]
is a typical example of $b[\lambda]$. 
Notice that if $E$ is of type (sp) then $[E, b_\lambda, \alpha] = 0$ in $W_\Gamma(K)$
for any $\lambda\in \Tw(E, \sigma)$.
In particular, the map $\partial_{E, \alpha}$ is zero. 
Let $E$ be of type (ur) or (rm). This means that $E$ is a field. 
Let $l$ denote the residue class field of $E$. The field $l$ is also the 
residue class field of the maximal unramified 
extension field of $E/K$, which is denoted by $L$.
The involution $\sigma$ on $E$ induces an involution $\sigma_l$ on $l$ 
(which may be trivial).
For any $x\in \cO_E$, the image of $x$ under $\cO_E\to l$ is denoted by $\bar{x}$. 
Let $\delta = v_E(\fD_{E/K})$ be the valuation of the different 
ideal $\fD_{E/K}$ of $E/K$.

\begin{lemma}\label{lem:AUL}
Let $\lambda\in (E^\sigma)^\times$.
\begin{enumerate}
\item If $v_E(\lambda) + \delta$ is even, set $n = -(v_E(\lambda) + \delta)/2$
and $\Lambda = \fm_E^n$. Then $\Lambda$ is an $\alpha$-stable unimodular lattice 
in $(E, b_\lambda)$.
\item If $v_E(\lambda) + \delta$ is odd, set $n = -(v_E(\lambda) + \delta - 1)/2$
and $\Lambda = \fm_E^n$. Then we have $\Lambda^\vee = \fm_E^{n-1}$, and in particular, 
$\Lambda$ is an $\alpha$-stable almost unimodular lattice in $(E, b_\lambda)$.  
Moreover, there is an isomorphism
$(\Lambda^\vee/\Lambda, \overline{b_\lambda}) \cong (l, b_{\bar{u}}, \bar{\alpha})$ 
as $k[\Gamma]$-forms,  
where $u\in \cO_L$ is the $\sigma$-invariant unit defined by
$u:=u_\lambda:=\Tr_{E/L}(\lambda \pi_K \pi_E^{n-1} \sigma(\pi_E^{n-1}))$ 
and $b_{\bar{u}}$ is the inner product defined by
\begin{equation}\label{eq:b_u}
 b_{\bar{u}}(\bar{x}, \bar{y}) 
= \Tr_{l/k}(\bar{u}\bar{x}\sigma_l(\bar{y})) \quad \text{for $x,y\in \cO_L$}.  
\end{equation}
\end{enumerate}
\end{lemma}
\begin{proof}
See Corollary 6.2 and Proposition 6.3 of \cite{BT20}.
\end{proof}

\begin{lemma}\label{lem:if_alpha_neq_pm1}
If $\bar{\alpha}$ is neither $1$ nor $-1$, 
then $\sigma_l$ is nontrivial and $l = k(\bar\alpha)$.
\end{lemma}
\begin{proof}
We have
\[
\bar{\alpha} \neq 1,-1
\iff \bar{\alpha}^2 - 1 \neq 0
\iff \bar{\alpha} - \bar{\alpha}^{-1} \neq 0
\iff \bar{\alpha} \neq \sigma_l(\bar{\alpha}).
\]
Let $\bar{\alpha} \neq 1, -1$. Then $\sigma_l$ is nontrivial, which implies that 
$l^{\sigma_l}$ is a proper subfield of $l$ and $[l:l^{\sigma_l}] = 2$.
If $[l:k(\bar\alpha)] \geq 2$ then we would get 
\[ 
[l^{\sigma_l} : k]
= [l:k]/2
\geq [l:k]/[l:k(\bar{\alpha})]
= [k(\bar{\alpha}) : k].
\]
This means that $k(\bar{\alpha}) \subset l^{\sigma_l}$ but this inclusion contradicts 
$\bar{\alpha} \neq \sigma_l(\bar{\alpha})$.
Hence we get $[l:k(\bar\alpha)] = 1$ and $l = k(\bar{\alpha})$.
\end{proof}

In the case (ur) we have:

\begin{proposition}\label{prop:partial_ur}
Assume that $E/E^\sigma$ is unramified. 
\begin{enumerate}
\item If $\bar\alpha \neq \pm1$ then 
$ \im \partial_{E, \alpha} 
= \{ 0, [k(\bar\alpha), b_1, \bar{\alpha}] \} \subset W_\Gamma(k)$
and $[k(\bar\alpha), b_1, \bar{\alpha}]$ has order $2$.
Here $b_1$ is the inner product defined by equation \eqref{eq:b_u} for $\bar{u} = 1$.
\item If $\bar\alpha = \pm1$ then
$ \im \partial_{E, \alpha} 
= \{ \omega\in W_\Gamma(k;\pm1) \mid \dim\omega \equiv 0 \bmod 2\}. $
\end{enumerate}
In particular $\dim\partial_{E,\alpha}(\lambda)\equiv 0 \bmod 2$ 
for any $\lambda\in \Tw(E, \sigma)$. 
\end{proposition}
\begin{proof}
Let $\lambda\in (E^\sigma)^\times$. 
Note that the class of $\lambda$ in $\Tw(E, \sigma)$ is uniquely determined by the 
parity of $v_E(\lambda) = v_{E^\sigma}(\lambda)$ since $E/E^\sigma$ is unramified.
If $v_E(\lambda) + \delta$ is even then $\partial_{E, \alpha}(\lambda) = 0$ 
by Lemma \ref{lem:AUL} \textup{(i)}.
Assume that $v_E(\lambda) + \delta$ is odd. Then 
$\partial_{E, \alpha}(\lambda) = [l, b_{\bar{u}}, \bar{\alpha}]$ 
by Lemma \ref{lem:AUL} \textup{(ii)}, where $u = u_\lambda$.

If $\bar\alpha \neq \pm1$,   
the involution $\sigma_l$ is nontrivial by Lemma \ref{lem:if_alpha_neq_pm1}.
This implies that all hermitian forms on $l$ (over the $k$-algebra $l$) 
are isomorphic to one another (see \cite[Example 10.1.6 (i)]{Sc85}), and their induced 
inner products over $k$ are isomorphic to $b_1$.
Thus, we have $2\partial_{E, \alpha}(\lambda) = 0$.
Furthermore, the $k[\Gamma]$-module $l$ is irreducible because $l = k(\bar\alpha)$ 
by Lemma \ref{lem:if_alpha_neq_pm1}. Hence 
$\partial_{E, \alpha}(\lambda) 
= [l, b_{\bar{u}}, \bar\alpha]
= [k(\bar\alpha), b_1, \bar{\alpha}]$
is not $0$ and has order $2$.

If $\bar\alpha = \pm1$ then 
\begin{equation}\label{eq:im_subset}
\im \partial_{E, \alpha} \subset
\{ \omega\in W_\Gamma(k;\pm1) \mid \dim\omega \equiv 0 \bmod 2\}  
\end{equation}
because
$[l:k] = [l:l'] [l':k] = 2[l':k] \in 2\bZ$, 
where $l'$ is the residue class field of $E^\sigma$.
If $\ch k = 2$ then the right-hand side of \eqref{eq:im_subset} is $\{ 0 \}$
and the statement is clear.
Let $\ch k \neq 2$. We have to prove that 
$\partial_{E, \alpha}(\lambda) = [l, b_{\bar{u}}, \bar\alpha]\in W_\Gamma(k; \pm 1)$
is nontrivial, and it is sufficient to show that 
$\disc b_{\bar{u}} \neq 1$ in $k^\times/k^{\times 2}$. We have
\[ \disc b_{\bar{u}} 
= \N_{l/k}(\bar{u}) \disc b_{1} 
= \N_{l'/k}\circ \N_{l/l'}(\bar{u}) \disc b_{1}
= \N_{l'/k}(\bar{u})^2 \disc b_{1}
= \disc b_{1} 
\]
in $k^\times/k^{\times 2}$.
One can verify that $(\disc b_1)^{\#k^\times/2} \neq 1$ to show that 
$\disc b_{1}$ is not a square. 
\end{proof}

In the case (rm) we have:

\begin{proposition}\label{prop:partial_rm}
Assume that $E/E^\sigma$ is ramified. Then $\bar\alpha = 1$ or $-1$ and
in particular $W_\Gamma(k; \bar\alpha) \cong W(k)$.
If $\ch k \neq 2$, then $v_E(\lambda) + \delta$ is odd for any 
$\lambda\in (E^\sigma)^\times$, and we have
\[ \im \partial_{E, \alpha}  
= \{ \omega\in W_\Gamma(k; \bar\alpha) \mid \dim \omega 
\equiv [l:k] \bmod 2 \}.
\]
If $\ch k = 2$ then for any 
$\lambda\in (E^\sigma)^\times$ we have
\[ \partial_{E, \alpha}(\lambda) 
= \begin{cases}
  0 & \text{if $[l:k]\delta$ is even}\\
  [\la 1 \ra] & \text{if $[l:k]\delta$ is odd.}
\end{cases} \]
Here $\la 1 \ra$ is a $1$-dimensional inner product over $k$
whose Gram matrix is the $1\times 1$ matrix $(1)$. 
\end{proposition}
\begin{proof}
See Lemma 6.5 and Propositions 6.6 and 6.7 of \cite{BT20}.
\end{proof}

We have now described the image of the map $\partial_{M,\alpha}$ 
in the case where $M$ is of rank one.  
In the general case, we obtain the following theorem. 

\begin{theorem}\label{th:imtheofTw_to_W}
Let $M$ be a free $E$-module of rank $m$. 
\begin{enumerate}
\item If $E$ is of type {\rm (sp)} then $\partial_{M,\alpha}$ is the zero map. 
\item If $E$ is of type {\rm (ur)} then 
\[\im \partial_{M,\alpha} = \begin{cases}
\{ 0, [k(\bar\alpha), b_1, \bar{\alpha}] \} & \text{if $\alpha \neq \pm1$} \\
\{ \omega\in W_\Gamma(k;\pm1) \mid \dim\omega \equiv 0 \bmod 2\}
& \text{if $\alpha = \pm1$.}
\end{cases}
\]
\item If $E$ is of type {\rm (rm)} then $\bar{\alpha} = 1$ or $-1$, 
and moreover, 
\begin{itemize}
\item if $\ch k \neq 2$ then 
$\im \partial_{M,\alpha} = 
\{ \omega\in W_\Gamma(k; \bar\alpha) \mid\allowbreak 
\dim \omega \equiv m[l:k] \bmod 2 \}$; 
\item if $\ch k = 2$ then 
for any $\lambda\in \Tw(E, \sigma)$ we have
\[ \partial_{M,\alpha}(\lambda) 
= \begin{cases}
  0 & \text{if $m[l:k]\delta$ is even}\\
  [\la 1 \ra] & \text{if $m[l:k]\delta$ is odd.}
\end{cases} \]
\end{itemize}
\end{enumerate}
\end{theorem}
\begin{proof}
(i) is clear. (ii) and (iii) follow from 
Propositions \ref{prop:partial_ur} and \ref{prop:partial_rm} respectively.
\end{proof}

\subsection{Even unimodular lattice with isometry}\label{ss:LT_EUL}
Let $F\in \cO_K[X]$ be a $*$-symmetric polynomial of even degree $2n$.
The polynomial $F$ can be expressed as
$F(X) = (X-1)^{n_+} (X+1)^{n_-} f(X)$, 
where $n_+, n_-\in \bZ_{\geq 0}$ and $f(X) \in \cO_K[X]$ with $f(1)f(-1) \neq 0$. 
Furthermore, the polynomial $f$ decomposes in $\cO_K[X]$ as  
\[ f(X) = \prod_{w\in \cW_{\rm sp}} (f_w(X) f_w^*(X))^{n_w} 
\times \prod_{w\in \cW'} f_w(X)^{n_w},  
\]
where $\cW_{\rm sp}$ and $\cW'$ are index sets, 
and $n_w$ is the multiplicity of $f_w$ in $f$. 
Moreover, each factor $f_w$ is irreducible and indexed as follows: 
if $w\in \cW_{\rm sp}$ then $f_w \neq f_w^*$; 
if $w\in \cW'$ then $f_w = f_w^*$. 

Now we define an algebra $M$. First, set $M^\pm := (K[X]/(X\mp 1))^{n_\pm}$. 
Next, for each $w\in \cW:=\cW_{\rm sp}\cup\cW'$, set 
\[ E_w := \begin{cases}
  K[X]/(f_w) & \text{if $f_w = f_w^*$} \\
  K[X]/(f_wf_w^*) & \text{if $f_w \neq f_w^*$}
\end{cases}
\]
and $M_w := (E_w)^{n_w}$. Finally, define
\[
M := M^+ \times M^- \times \prod_{w\in \cW} M_w.
\]
We denote by $\alpha$ the image of $X$ in $M$, and by $\alpha_w$ the
image of $X$ in $M_w$.
It is obvious that the linear map $\alpha:M\to M$ is a semisimple 
automorphism with characteristic polynomial $F$.
In addition, the $K$-algebra $M$ has an involution $\sigma$ defined by 
$\alpha \mapsto \alpha^{-1}$.
It is clear that
\[ \sigma M^\pm = M^\pm, 
 \quad \sigma E_w = E_w \text{ for $w\in \cW$},
\]
and $\sigma$ acts on $M^\pm$ as the identity and on $E_w$ nontrivially. 
The restriction of $\sigma$ to a $\sigma$-invariant subalgebra 
will also be denoted by $\sigma$. 

We remark that there is a decomposition 
$\cW = \cW_{\rm sp} \sqcup \cW_{\rm ur} \sqcup \cW_{\rm rm}$
(or $\cW' = \cW_{\rm ur} \sqcup \cW_{\rm rm}$), where
\begin{equation}\label{eq:spurrm}
\begin{alignedat}{4}
&\cW_{\rm sp} & = 
& \{ w\in \cW \mid f_w \neq f_w^* 
&\text{ (and } & \text{$E_w$ is of type {\rm (sp)}}) &\}, \\
&\cW_{\rm ur} & := 
& \{ w\in \cW \mid f_w = f_w^* 
&\text{ and } & \text{$E_w$ is of type {\rm (ur)}} &\}, \\
&\cW_{\rm rm} & := 
& \{ w\in \cW \mid f_w = f_w^* 
&\text{ and } & \text{$E_w$ is of type {\rm (rm)}} &\}.  
\end{alignedat}
\end{equation}
Furthermore, by Theorem \ref{th:imtheofTw_to_W} (iii), 
if $\ch k \neq 2$ then $\cW_{\rm rm}$ decomposes as
$\cW_{\rm rm} = \cW_+\sqcup \cW_-$, where
$\cW_\pm :=\allowbreak \{ w\in \cW_{\rm rm} \mid \bar\alpha_w = \pm1 \}$.
In this subsection we write $v = v_K$ simply.

\begin{lemma}\label{lem:dimpartial}
Let $w\in \cW$. For any $\lambda_w\in \Tw(E_w, \sigma)$ we have: 
\begin{enumerate}
\item If $w\in \cW_{\rm sp}$ then 
$\dim\partial_{M_w, \alpha_w}(\lambda_w) 
\equiv v(f_w(1)f_w^*(1)) 
\equiv v(f_w(-1)f_w^*(-1)) 
\equiv 0 \bmod 2$. 
\item If $w\in \cW_{\rm ur}$ then 
$\dim\partial_{M_w, \alpha_w}(\lambda_w) 
\equiv v(f_w(1)) 
\equiv v(f_w(-1)) 
\equiv 0 \bmod 2$. 
\item Suppose that $w\in \cW_{\rm rm}$. 
\begin{itemize}
\item If $\ch k \neq 2$ and $\bar\alpha_w = \pm 1$ then
$\dim\partial_{M_w, \alpha_w}(\lambda_w) \equiv n_w v(f_w(\pm 1)) \bmod 2$.
\item If $\ch k = 2$ then
$\dim\partial_{M_w, \alpha_w}(\lambda_w) 
\equiv n_w v(f_w(1)) + n_w v(f_w(-1)) \bmod 2$.
\end{itemize}
\end{enumerate}
In particular, if $\ch k \neq 2$ then  
\[ \sum_{w\in\cW_{\rm \pm}}\dim\partial_{M_w, \alpha_w}(\lambda_w) 
\equiv v(f(\pm1)) \mod 2 \]
for any $(\lambda_w)_{w\in \cW_{\rm rm}}$.
\end{lemma}
\begin{proof}
\textup{(i)} is clear. If $w\in \cW_{\rm ur}$ then
$\dim\partial_{M_w, \alpha_w}(\lambda_w) \equiv 0 \bmod 2$ by 
Theorem \ref{th:imtheofTw_to_W}, and
$v(f_w(\pm1)) = [l_w:k]\,v_{E_w}(1\pm\alpha_w) \equiv 0 \bmod 2$, 
where $l_w$ is the residue class field of $E_w$.
Suppose that $w\in \cW_{\rm rm}$. 
If $\ch k \neq 2$ and $\bar\alpha_w = \pm 1$ 
then $\dim\partial_{M_w, \alpha_w}(\lambda_w) \equiv\allowbreak n_w[l_w:k] \bmod 2$
by Theorem \ref{th:imtheofTw_to_W}. 
Because $[l_w:k] \equiv v(f_w(\pm 1)) \bmod 2$ by \cite[Lemma 6.8]{BT20}, 
we obtain the desired congruence. 
Similarly, Theorem \ref{th:imtheofTw_to_W}
and \cite[Lemma 6.8]{BT20} show the desired congruence in the case $\ch k = 2$.
\end{proof}

The following proposition is a reformulation of a part of 
\cite[Proposition 4.3]{Ba21}.

\begin{proposition}\label{prop:detailedneq2}
Let $\ch k \neq 2$. Assume that $v(F(1)) \equiv v(F(-1)) \equiv 0 \bmod 2$.
Then there exists an inner product $b$ on $M$ such that 
$(M,b)$ contains an $\alpha$-stable unimodular $\cO_K$-lattice. 
Furthermore, 
\begin{itemize}
\item if $F(\pm 1) = 0$ then such an inner product can be chosen to
satisfy $\det M^\pm = u_\pm f(\pm 1)$ for any given $u_\pm \in \cO_K^\times$; and 
\item if $\cW_{\rm rm} = \emptyset$ then such an inner product 
can be chosen for $(M_w,b_w)$ to contain an $\alpha$-stable unimodular $\cO_K$-lattice
for each $w\in \cW$. Here $b_w$ is the restriction of $b$ to $M_w$.
\end{itemize}
\end{proposition}
\begin{proof}
If $F(\pm 1) = 0$ we choose an inner product $b^\pm$ on $M^\pm$ whose 
Gram matrix is $\diag(u_\pm f(\pm1), 1, \ldots, 1)$.
For each $w\in \cW_{\rm sp}\cup \cW_{\rm ur}$, choose $\lambda_w\in \Tw(E_w,\sigma)$
satisfying $\partial_{M_w, \alpha_w}(\lambda_w) = 0$.
This is possible by Theorem \ref{th:imtheofTw_to_W}. 
For any $(\lambda_w)_{w\in \cW_{\pm}}$ we have by Lemma \ref{lem:dimpartial}
\[ \begin{split}
\dim\partial\left[\bigoplus_{w\in\cW_\pm}(M_w, b[\lambda_w], \alpha_w)\right]
\equiv v(f(\pm 1))
\equiv v(\dim\partial[M^\pm, b^\pm ,\pm1]) \mod 2.
\end{split}\]
Then we can choose $(\lambda_w)_{w\in \cW_{\pm}}$ satisfying
$\partial[\bigoplus_{w\in\cW_\pm}(M_w, b[\lambda_w], \alpha_w)]
=\allowbreak -\partial[M^\pm, b^\pm ,\pm1]$ by Theorem \ref{th:imtheofTw_to_W}.
The inner product 
\[b:= b^+\oplus b^-\oplus \bigoplus_{w\in \cW}b[\lambda_w]\] 
on $M$ satisfies 
$\partial[M,b,\alpha] = 0$. This implies that $(M,b)$ contains an 
$\alpha$-stable unimodular $\cO_K$-lattice by Theorem \ref{th:ULvanish}.
The latter part of this proposition is obvious by the above construction.
\end{proof}

We consider an analog of Proposition \ref{prop:detailedneq2} when $K = \bQ_2$. 
Our purpose is to show Theorem \ref{th:detailedQ2}, which is a generalization 
of \cite[Theorem 5.1]{Ba21} to the case where $F$ is $*$-symmetric.
We denote by $\sn(t)$ the \textit{spinor norm} of an isometry $t$, 
see \cite[\S 55]{OM73} or \cite{Za62} for definition.

\begin{lemma}\label{lem:val_of_sp_norm}
Let $(\Lambda, b)$ be an even unimodular $\bZ_2$-lattice.
For any isometry $t:\Lambda \to \Lambda$,
we have $v(\sn(t)) \equiv 0 \bmod 2$ if $\det t = 1$, and 
$v(\sn(t)) \equiv 1 \bmod 2$ if $\det t = -1$.
\end{lemma}
\begin{proof}
Let $t$ be an isometry on $\Lambda$.
If $\det t = 1$ then $v(\sn(t)) \equiv 0 \bmod 2$ by
\cite[Proposition 8.6]{BT20}.
Let $\det t = -1$. 
We can choose $r\in \Lambda$ satisfying $v(b(r,r)) = 1$ by 
Proposition \ref{prop:Z2-EUL}. Let $s_r$ denote the reflection defined by $r$.
Then $\det(s_r\circ t) = 1$ and hence $v(\sn(s_r\circ t)) \equiv 0 \bmod 2$.
Because $v(\sn(s_r)) \equiv 1$ we get $v(\sn(t)) \equiv 1 \bmod 2$.
\end{proof}

\begin{proposition}\label{prop:NScdforZ2EUL}
Let $(V,b)$ be an inner product space over $\bQ_2$, and $t$ an isometry on $V$.
There exists a $t$-stable even unimodular lattice of discriminant $1$ in $V$
if and only if the following conditions hold:
\begin{enumerate}
\item $V$ contains a $t$-stable unimodular lattice.
\item $V$ contains an even unimodular lattice of discriminant $1$.
\item $v(\sn(t))\equiv \begin{cases}
  0 \bmod 2 & \text{if $\det t = 1$} \\
  1 \bmod 2 & \text{if $\det t = -1$}. 
\end{cases}$
\end{enumerate}
\end{proposition}
\begin{proof}
The basic idea of the proof is similar to that of \cite[Theorem 8.1]{BT20}. 
If there exists a $t$-stable even unimodular lattice of discriminant $1$ in $V$, 
then the conditions (i) and (ii) are obvious, and (iii) holds by 
Lemma \ref{lem:val_of_sp_norm}. 

Let the three conditions hold. There exists a $t$-stable unimodular lattice
$\Lambda_0$ in $V$ by (i). If $\Lambda_0$ is even then we are done. 
Assume that $\Lambda_0$ is odd. We can take an even unimodular lattice 
$\Lambda_1$ of discriminant $1$ in $V$ by (ii). 
Proposition \ref{prop:Z2-EUL} implies that there exist vectors 
$2e, f\in \Lambda_1$ satisfying
\[ b(2e, 2e) = b(f,f) = 0 \text{ and } b(2e,f) = 1. \]
Set $H:=(\bZ_2(2e) + \bZ_2 f)^\perp \subset \Lambda_1$ and 
$\Lambda_0':= H + \bZ_2(e+f) + \bZ_2(e-f)$ in $V$.
Because $\Lambda_0'$ is an odd unimodular lattice in $V$, we may assume that 
$\Lambda_0' = \Lambda_0$ by \cite[Theorem 93:29]{OM73}. 
We now define lattices $\Lambda_2$ and $\Lambda$ in $V$ by 
$\Lambda_2 := H + \bZ_2 e + \bZ_2(2f)$ and by 
$\Lambda := H + \bZ_2(2e) + \bZ_2(2f)$ respectively.
Then $\Lambda_2$ is an even unimodular lattice different from $\Lambda_1$, 
and $\Lambda$ is contained in $\Lambda_0, \Lambda_1$ and $\Lambda_2$. 
We remark that $\Lambda$ can be written as 
$\Lambda = \{ x\in \Lambda_0 \mid b(x,x)\in 2\bZ_2 \}$, so $t$ preserves $\Lambda$. 
There is no unimodular lattice containing $\Lambda$ other than 
$\Lambda_0, \Lambda_1$ and $\Lambda_2$ because there is a natural one-to-one 
correspondence between the $\bZ_2$-valued lattices containing $\Lambda$ and 
the isotropic subgroups of 
$\Lambda^\vee/\Lambda\cong \bZ/2\bZ\times\bZ/2\bZ$.
Hence, we have $t\Lambda_1 = \Lambda_1$ or $t\Lambda_1 = \Lambda_2$. 
We need to show that $t\Lambda_1 = \Lambda_1$. 
Suppose to the contrary that $t\Lambda_1 = \Lambda_2$. 
Let $s:V\to V$ denote the reflection defined by $e-f$. 
Because $s$ maps $\Lambda_2$ to $\Lambda_1$, the composition $s\circ t$ 
preserves $\Lambda_1$. On the other hand, by the equation 
$\sn(t) = b(e-f,e-f) = -1$ and the assumption (iii), we would have 
\[ 
v(\sn(s\circ t))
\equiv v(\sn(t))
\equiv \begin{cases}
  0 & \text{if $\det t = 1$}\\
  1 & \text{if $\det t = -1$}
\end{cases}
\equiv \begin{cases}
  0 & \text{if $\det (s\circ t) = -1$}\\
  1 & \text{if $\det (s\circ t) = 1$}
\end{cases}
\]
mod $2$. However, this contradicts Lemma \ref{lem:val_of_sp_norm}. 
Therefore, we have $t\Lambda_1 = \Lambda_1$. 
\end{proof}

Now, we make a formula for the valuation of the spinor norm of an isometry. 

\begin{lemma}
Let $t$ be an isometry on an inner product space $(V,b)$ over $\bQ_2$ of 
even dimension, and let $F(X) = (X-1)^{n_+}(X+1)^{n_-}f(X)$ be its characteristic 
polynomial where $n_+, n_-\in\bZ_{\geq 0}$ and $f(1)f(-1)\neq 0$. 
Let $V(-1;t)$ denote the eigenspace of $t$ corresponding to $-1$. 
Then we have
\begin{equation}\label{eq:spinorvaluation}
v(\sn(t)) \equiv v(\det b|_{V(-1;t)}) + v(f(-1)) \mod 2.  
\end{equation}
\end{lemma}
\begin{proof}
The Zassenhaus formula \cite{Za62} implies that 
\[\sn(t) = \det b|_{V(-1;t)} \cdot 
\det \left(\frac{1 + t}{2}\middle|{V(-1;t)^\perp} \right) 
\quad\text{in $\bQ_2^\times/\bQ_2^{\times 2}$, }\]
and we have
\[ \det \left(\frac{1 + t}{2}\middle|{V(-1;t)^\perp} \right)
= (-2)^{\dim(V) -n_-}\det(-1-t |_{V(-1;t)^\perp})
= (-2)^{n_- + n_+}  f(-1)
= f(-1)
\]
in $\bQ_2^\times/\bQ_2^{\times 2}$ since $n_+ + n_-$ is even. 
Therefore, we get equation \eqref{eq:spinorvaluation}. 
\end{proof}

\begin{theorem}\label{th:detailedQ2}
Let $K = \bQ_2$, and let $F\in \bZ_2[X]$ be a $*$-symmetric polynomial of even 
degree $2n$. Assume that 
\begin{itemize}
\item[\textup{(a)}] $v(F(1)) \equiv v(F(-1)) \equiv 0 \bmod 2$; and
\item[\textup{(b)}] if $F(1)F(-1)\neq 0$ then $(-1)^n F(1)F(-1) = 1$ in 
$\bQ_2^\times/\bQ_2^{\times 2}$.
\end{itemize}
Then there exists an inner product $b$ on $M$ such that 
$(M,b)$ contains an $\alpha$-stable even unimodular $\bZ_2$-lattice. 
Furthermore, if $F(1) = F(-1) = 0$ then such an inner product can be chosen to satisfy 
\[\det M^\pm = \begin{cases}
  u_\pm f(\pm 1) & \text{if $n_+$ is even}\\ 
  2u_\pm f(\pm 1) & \text{if $n_+$ is odd}
\end{cases} \]
for any given $u_+, u_- \in \cO_K^\times$ such that $u_+ u_- = (-1)^n$.
\end{theorem}
\begin{proof}
We take an inner product $b$ on $M$ as follows. First, set 
\[D_\pm = \begin{cases}
  u_\pm f(\pm 1) & \text{if $n_+$ is even}\\ 
  2u_\pm f(\pm 1) & \text{if $n_+$ is odd}
\end{cases} \]
and take inner products $b^+$ on $M^+$ and $b^-$ on $M^-$ satisfying
\begin{equation*}
\det b^+ = \begin{cases}
  (-1)^n f(1)f(-1) & \text{if $F(1) = 0$ and $F(-1)\neq 0$}\\ 
  D_+ & \text{if $F(1) = F(-1) = 0$},
\end{cases}
\end{equation*}
\begin{equation}\label{eq:detb-}
\det b^- = \begin{cases}
  (-1)^n f(1)f(-1) & \text{if $F(1) \neq 0$ and $F(-1) = 0$}\\ 
  D_- & \text{if $F(1) = F(-1) = 0$}.
\end{cases}
\end{equation}
Next, for each $w\in \cW_{\rm sp} \cup \cW_{\rm ur}$ we fix
$\lambda_w \in \Tw(E_w, \sigma)$ satisfying 
$\partial_{M_w, \alpha_w}(\lambda_w) = 0$.
This is possible by Theorem \ref{th:imtheofTw_to_W}.
Furthermore, we choose $\lambda_w \in \Tw(E_w, \sigma)$ arbitrarily for each 
$w\in \cW_{\rm rm}$ and define an inner product on $M$ by 
$b := b^+\oplus b^-\oplus \bigoplus_{w\in\cW}b[\lambda_w]$. 
Notice that $\alpha$ is an isometry with respect to $b$. 
\smallskip

\textit{\textup{Claim 1:} For any $b$ chosen as above, 
the inner product space $(M,b)$ with the isometry $\alpha$ 
satisfies the conditions \textup{(i)} and \textup{(iii)} 
in Proposition \textup{\ref{prop:NScdforZ2EUL}}.}

For \textup{(i)}, it is enough to show that $\partial[M,b,\alpha] = 0$ by Theorem 
\ref{th:ULvanish}. We have
\[ \dim \partial[M^+\oplus M^-, b^+\oplus b^-, \alpha]
\equiv v(\det (b^+\oplus b^-))
\equiv v((-1)^nf(1)f(-1))
\equiv v(f(1)f(-1))
\]
mod $2$. On the other hand, by Lemmas 
\ref{lem:dimpartial} and \ref{lem:det_iso_formula} we have 
\[\begin{split}
&\dim \partial\left[\bigoplus_{w\in\cW_{\rm rm}} M_w, 
\bigoplus_{w\in\cW_{\rm rm}} b[\lambda_w], \alpha
\right]
\equiv \dim \partial\left[\bigoplus_{w\in\cW} M_w, 
\bigoplus_{w\in\cW} b[\lambda_w], \alpha
\right]\\
&\quad \equiv v\left(\det \bigoplus_{w\in\cW} b[\lambda_w]\right)
\equiv v(f(1)f(-1)) \mod 2.
\end{split} \]
Therefore 
$\dim\partial[M^+\oplus M^- \oplus \bigoplus_{w\in\cW_{\rm rm}} M_w, 
b^+\oplus b^- \oplus \bigoplus_{w\in\cW_{\rm rm}} b[\lambda_w], \alpha] \equiv 0 
\bmod 2$ and hence
\[\begin{split}
&\partial\left[M^+\oplus M^- \oplus \bigoplus_{w\in\cW} M_w, 
b^+\oplus b^- \oplus \bigoplus_{w\in\cW} b[\lambda_w], \alpha \right] \\
&= \partial\left[M^+\oplus M^- \oplus \bigoplus_{w\in\cW_{\rm rm}} M_w, 
b^+\oplus b^- \oplus \bigoplus_{w\in\cW_{\rm rm}} b[\lambda_w], \alpha \right] \\ 
&\quad + \sum_{w\in\cW_{\rm sp} \cup \cW_{\rm ur}} 
\partial\left[M_w, b[\lambda_w], \alpha_w \right] \\
&= 0. 
\end{split}
\]

Let us show \textup{(iii)}. If $F(-1) \neq 0$ then 
$v(\sn(\alpha)) \equiv v(f(-1))\equiv 0$ mod $2$
by \eqref{eq:spinorvaluation} and the assumption \textup{(a)}.
Let $F(-1) = 0$. If $F(1)\neq 0$ then 
\[v(\sn(\alpha))
\equiv v((-1)^n f(1)f(-1)) + v(f(-1))
\equiv v(f(1))
\equiv 0 \mod 2\]
by \eqref{eq:spinorvaluation}, \eqref{eq:detb-}, and the assumption \textup{(a)}.
If $F(1) = 0$ then 
\[ v(\sn(\alpha)) \equiv \begin{cases}
  v(D_-) + v(f(-1)) \equiv 0 & \text{if $n_+$ is even} \\
  v(D_-) + v(f(-1)) \equiv 1 & \text{if $n_+$ is odd}
\end{cases}
\]
mod $2$ by \eqref{eq:spinorvaluation} and \eqref{eq:detb-}. 
Hence, the condition \textup{(iii)} holds in any case, and Claim 1 has now been proved. 
\smallskip

\textit{\textup{Claim 2:} If $b^+$ and $b^-$, and $\lambda_w \in \Tw(E_w, \sigma)$ 
for each $w\in \cW_{\rm rm}$ are suitably chosen, then $(M,b)$ satisfies the condition 
\textup{(ii)} in Proposition \textup{\ref{prop:NScdforZ2EUL}}.}

In general, there is a unique $\eta_m \in \{0, 1\}$ for each $m\in \bZ_{\geq 0}$ 
such that any inner product space over $\bQ_2$ of dimension $2m$, discriminant $1$, 
and Hasse-Witt invariant $\eta_m$ contains an even unimodular $\bZ_2$-lattice. 
This is a consequence of Proposition \ref{prop:Z2-EUL}.
Assume that $\cW_{\rm rm} \neq \emptyset$, and let $w_0 \in \cW_{\rm rm}$.
Choose $\widehat\lambda_{w_0}\in \Tw(E_{w_0}, \sigma)$ which is different 
from $\lambda_{w_0}$, and define 
$\widehat{b} := b^+ \oplus b^- \oplus b[\widehat\lambda_{w_0}] \oplus 
\allowbreak\bigoplus_{w\neq w_0}b[\lambda_w]$.
Because $\epsilon_2(b) \neq \epsilon_2(\widehat{b})$,
we have $\epsilon_2(b) = \eta_n$ or $\epsilon_2(\widehat{b}) = \eta_n$.
This means that the condition \textup{(ii)} holds for $(M,b)$ or $(M,\widehat{b})$. 

Let $\cW_{\rm rm} = \emptyset$. 
Since $ \partial \left[\bigoplus_{w\in \cW} M_w, 
\bigoplus_{w\in \cW} b[\lambda_w], \alpha \right] = 0$, 
the space 
$(\bigoplus_{w\in \cW} M_w, \bigoplus_{w\in \cW} b[\lambda_w])$
contains a unimodular lattice, 
and the unimodular lattice is even (see the proof of \cite[Proposition 9.1]{BT20}).
Therefore it is sufficient to show that $(M^+\oplus M^-, b^+\oplus b^-)$ contains 
an even unimodular lattice. 

\textit{Case\textup{\Rnum{1}}. $n_+>2$ or $n_->2$.}
Suppose that $n_+>2$. Then there exists an inner product $\widehat{b}^+$ on $M^+$ with 
$\det \widehat{b}^+ = \det b^+$ and $\epsilon_2(\widehat{b}^+) \neq \epsilon_2(b^+)$. 
Because $\epsilon_2(b^+ \oplus b^-) = \eta_{n_+ + n_-}$ or 
$\epsilon_2(\widehat{b}^+ \oplus b^-) = \eta_{n_+ + n_-}$, 
either $(M^+\oplus M^-, b^+\oplus b^-)$ or $(M^+\oplus M^-, \widehat{b}^+\oplus b^-)$
contains an even unimodular lattice. 
Similarly, if $n_- > 2$ then $(M^+\oplus M^-, b^+\oplus b^-)$ contains an even unimodular 
lattice for a suitable $b^-$. 

\textit{Case\textup{\Rnum{2}}. $(n_+, n_-) = (2,2)$.}
If $D_+ \neq -1$ or $D_- \neq -1$ in $\bQ_2^\times/\bQ_2^{\times2}$ then we can choose 
$b_+$ or $b_-$ to get $\epsilon_2(b_+ \oplus b_-) = \eta_{4}$ as in Case\Rnum{1}. 
If $D_+ = D_- = -1$ in $\bQ_2^\times/\bQ_2^{\times2}$ then $b^+$ and $b^-$ are 
isomorphic to the hyperbolic plane and contain even unimodular lattices respectively.

\textit{Case\textup{\Rnum{3}}. $(n_+, n_-) = (2,0)$ or $(0,2)$.}
A similar proof of Case\Rnum{2} works, and $b^+$ or $b^-$ contains
an even unimodular lattice if we choose $b^+$ or $b^-$ suitably. 

\textit{Case\textup{\Rnum{4}}. $(n_+, n_-) = (1,1)$.}
Since $b$ has discriminant $1$ and 
$(\bigoplus_{w\in \cW} M_w, \bigoplus_{w\in \cW} b[\lambda_w])$
contains an even unimodular lattice, we have 
$\disc(b^+\oplus b^-) = \disc(\bigoplus_{w\in \cW} b[\lambda_w]) = 1$ or $-3$.
If $\disc(b^+\oplus b^-) = 1$ then $b^+\oplus b^-$ is isomorphic to the
hyperbolic plane and contains an even unimodular lattice.
Let $\disc(b^+\oplus b^-) = -3$. 
Lemma \ref{lem:dimpartial} implies that $v(f(1)) \equiv 0\bmod 2$ 
since $\cW_{\rm rm} = \emptyset$.
The Hasse-Witt invariant of $b^+\oplus b^-$ can be calculated as
\[
\epsilon_2(b^+\oplus b^-)
= (D_+, D_-)
= (D_+, -D_+D_-)
= (2 u_+ f(1), -3)
= 1,
\]
and this means that $b^+\oplus b^-$ is isomorphic to the lattice $V$ in 
Proposition \ref{prop:Z2-EUL}.
Thus, the space $(M^+\oplus M^-, b^+\oplus b^-)$ 
contains an even unimodular lattice. 

In any case, the space $(M^+\oplus M^-, b^+\oplus b^-)$ contains an even unimodular 
lattice if we choose $b^+$ and $b^-$ suitably. 
This completes the proof of Claim 2. 
\smallskip

Claims 1 and 2 mean that $(M, b)$ with the isometry $\alpha$ 
satisfies the conditions \textup{(i)--(iii)} in Proposition \ref{prop:NScdforZ2EUL}  
for a suitable inner product $b$, 
which implies that $(M,b)$ contains an $\alpha$-stable even unimodular lattice. 
The latter part of this theorem is obvious by the above construction.
\end{proof}

Finally, we prove Theorem \ref{th:localcdforEUL} as promised. 
\smallskip
\\
\textit{Proof of Theorem \textup{\ref{th:localcdforEUL}}.}
Let $(\Lambda,b)$ be an even unimodular lattice of rank $2n$ and discriminant $1$
having a semisimple isometry $t$ with characteristic polynomial $F$. 
If $\det t = -1$ i.e., $F$ is $-1$-symmetric, we have 
$F(1) = F(-1) = 0$. Thus, the conditions \textup{(a)} and \textup{(b)}
are clear. Let $\det t = 1$. We can write $F$ as
$F(X) = (X-1)^{n_+}(X+1)^{n_-} f(X)$ using even integers $n_+,n_-$ and
$f\in \bZ_2[X]$ with $f(1)f(-1) \neq 0$.  
Equation \eqref{eq:spinorvaluation} and Proposition \ref{prop:NScdforZ2EUL}
imply that 
\begin{equation}\label{eq:detb-f-1}
v(\det b|_{V(-1;t)}) + v(f(-1)) \equiv 0 \mod 2.
\end{equation}
On the other hand, we have
\[v(\det b|_{V(1;t)}) + v(\det b|_{V(-1;t)}) + v(f(1)) + v(f(-1))
\equiv v(\det b) \equiv 0 \mod 2
\]
by Lemma \ref{lem:det_iso_formula}, and hence 
\begin{equation}\label{eq:detb+f+1}
v(\det b|_{V(1;t)}) + v(f(1)) \equiv 0 \mod 2.
\end{equation}
If $F(-1) \neq 0$ then $v(F(-1)) \equiv v(f(-1)) \equiv 0 \bmod 2$
by equation \eqref{eq:detb-f-1}. 
Similarly, if $F(1) \neq 0$ then $v(F(-1)) \equiv v(f(-1)) \equiv 0 \bmod 2$
by equation \eqref{eq:detb+f+1}. 
Therefore, we get the condition \textup{(a)}.
Let $F(1)F(-1) \neq 0$. Then we have 
\[ (-1)^nF(1)F(-1) = (-1)^n \det b = \disc b = 1  
\quad\text{in $\bQ_2^\times/\bQ_2^{\times 2}$} \]
by Lemma \ref{lem:det_iso_formula}.
This is the condition \textup{(b)}.
The converse follows from Theorem \ref{th:detailedQ2}. \qed

\section{Local-global principle}\label{sec:LGP}
Let $F\in \bZ[X]$ be a $*$-symmetric polynomial of even degree $2n$,
and let $r,s\in \bZ_{\geq 0}$ be non-negative integers such that 
$r + s = 2n$ and $r \equiv s \bmod 8$.
Assume that the condition \eqref{eq:Sgcd} holds, and 
let $\iota \in \Idx_{r,s}(F)$ be an index map. 
We refer to an isometry with characteristic polynomial $F$ and index $\iota$ 
as an $(F,\iota)$-isometry for short. 
In this section, we establish a necessary and sufficient condition
for an even unimodular lattice of signature $(r, s)$ having a semisimple 
$(F, \iota)$-isomety to exist (Theorem \ref{th:LGP}). 

We start by constructing a vector space having a semisimple automorphism with
characteristic polynomial $F$, as in \S\ref{ss:LT_EUL}. 
Let $F_i$ be the type $i$ component of $F$ for $i = 0,1,2$: $F = F_0F_1F_2$. 
The product $F_1F_2$ is sometimes abbreviated to $F_{12}$. 
The symbol $I_i$ denotes the set of irreducible factors of $F_i$. 
Let $n_+$ and $n_-$ be the multiplicities of $(X-1)$ and $(X+1)$ in $F_0$ 
respectively, and $n_f$ be the multiplicity of $f\in I_1 \cup I_2$ in $F_{12}$.
Set 
\begin{equation*}
\begin{split}
&M^\pm := [\bQ[X]/(X\mp 1)]^{n_\pm}, \quad M^0 := M^+ \times M^-, \\
&E^f := \begin{cases}
  \bQ[X]/(f) & \text{for $f\in I_1$}\\
  \bQ[X]/(ff^*) & \text{for $f\in I_2$, }
\end{cases}
\quad M^f := (E^f)^{n_f} \text{ for $f\in I_1\cup I_2$}, \\
&M^1 := \prod_{f\in I_1} M^{f},
 \quad M^2 := \prod_{\{f,f^*\}\subset I_2} M^f, 
 \quad\text{and}\quad M := M^0 \times M^1 \times M^2 .
\end{split}
\end{equation*}
Let $\alpha$ denote the image of $X$ in $M$,  
and $\sigma$ the involution defined by $\alpha \mapsto \alpha^{-1}$. 
If we regard $M$ as a $\bQ$-vector space, the $\bQ$-linear map $\alpha:M\to M$ 
is a semisimple automorphism with characteristic polynomial $F$.
We will consider when $M$ admits an inner product such that $\alpha$ becomes an 
isometry preserving an even unimodular lattice.

\begin{notation}
We use the following notations.
\begin{itemize}
\item The set of places of $\bQ$ is denoted by $\cV$.
\item The set of places of $E^{f,\sigma} := (E^f)^\sigma$ above $v\in\cV$
is denoted by $\cW(f;v)$ for $f\in I_1\cup I_2$.  
\item Let $K$ be an algebraic number field, and $v$ its place. 
For a $K$-algebra $A$, we write $A_v$ for $A\otimes_K K_v$, 
where $K_v$ is the completion of $K$ with respect to $v$. 
For a $K$-algebra $A^\bullet$ with a superscript, such as $M^\pm, M^f$, or $E^f$, 
we abbreviate $(A^\bullet)_v$ to $A_v^\bullet$. 
\end{itemize}
\end{notation}

Let $I$ denote the set of irreducible $*$-symmetric factors of $F$, that is, 
$I = I_0\cup I_1$.  
If $v\in \cV$ is a place and $b_v$ is an inner product on $M_v$, then 
the symbol $b_v^f$ denotes the restriction of $b_v$ to $M_v^f = M^f \otimes \bQ_v$ 
for each $f\in I$. Here, we understand that $M^{X\mp1} = M^{\pm}$. 
We will show the following theorem.

\begin{theorem}\label{th:LGP}
Let $F\in\bZ[X], r,s\in \bZ_{\geq0}$ and $\iota\in\Idx_{r,s}(F)$ be 
as stated at the beginning of this section. The following are equivalent:
\begin{enumerate}
\item There exists an even unimodular lattice of signature $(r,s)$ having 
a semisimple $(F, \iota)$-isometry.
\item There exists a family $\{ b_v \}_{v\in \cV}$ of inner products on $M_v$
such that each $b_v$ has the properties 
\begin{equation}\label{eq:P1}
\text{$\alpha:M_v \to M_v$ is an isometry with respect to $b_v$\,;} \tag{P1}
\end{equation}
\begin{equation}\label{eq:P2}
\begin{split}
&\text{if $v\neq \infty$ then there exists an $\alpha$-stable
even unimodular $\bZ_v$-lattice in $(M_v, b_v)$, and}\\
&\text{if $v = \infty$ then the isometry $\alpha$ on $(M_\infty, b_\infty)$ 
has index $\iota$\,;}
\end{split}\tag{P2}
\end{equation}
\begin{equation}\label{eq:P3}
\det b^{X\mp1}_v = \begin{cases}
(-1)^{(n_\pm - \iota(X\mp1))/2} \cdot |F_{12}(\pm 1)|  & \text{if $n_+$ is even} \\
(-1)^{(n_\pm - \iota(X\mp1))/2} \cdot 2 \, |F_{12}(\pm 1)| & \text{if $n_+$ is odd}
\end{cases}
\quad\text{in $\bQ_v^\times/\bQ_v^{\times 2}$,}
\tag{P3}
\end{equation}
and that 
almost all $\epsilon_v(b_v^f)$ equal $0$ 
(i.e. $\#\{ (v,f) \in \cV\times I \mid \epsilon_v(b_v^f) = 1 \} < \infty$) and
\[ 
\sum_{v\in \cV} \epsilon_v(b_v^f) = 0 \quad 
\text{ for all $f\in I$}.
\]
\end{enumerate}
\end{theorem}

To prove this theorem, we need to consider localizations of each $M^f$. 
Let $v\in \cV$. We have
\[ \begin{split}
M^\pm_v = [\bQ[X]/(X\mp 1)]^{n_\pm} \otimes \bQ_v = [\bQ_v[X]/(X\mp 1)]^{n_\pm}, 
\end{split}\]
and 
\begin{equation}\label{eq:M^f_w}
M^f_w 
= M^f \otimes_{E^{f,\sigma}} E^{f,\sigma}_w 
= [E^f \otimes_{E^{f,\sigma}} E^{f,\sigma}_w]^{n_f}
= [E^f_w]^{n_f}
\end{equation}
for $f\in I_1\cup I_2$ and $w\in \cW(f;v)$. 
Note that there exists the canonical isomorphism
between the completion $(E^{f,\sigma})_w$ and
the fixed subalgebra $(E^f_w)^\sigma$ of 
$E^f_w = E^f\otimes_{E^{f,\sigma}} (E^{f,\sigma})_w$.
They are identified and denoted by $E_w^{f,\sigma}$. 

\begin{lemma}
Let $p$ be a prime. We have
\[ M^1_p = \prod_{f\in I_1} \prod_{w\in \cW(f;p)} M^f_w, \quad
M^2_p = \prod_{\{f, f^*\}\subset I_2} \prod_{w\in \cW(f;p)} M^f_w \]
and each $M^f_w$ is the direct product of $n_f$ copies of $E^f_w$.
Furthermore, each $E^f_w$ satisfies one of 
{\rm (sp)}, {\rm (ur)} and {\rm (rm)} in \textup{\S\ref{ss:imageofmaps}}.
\end{lemma}
\begin{proof}
By equation \eqref{eq:M^f_w} the algebra $M^f_w$ is the direct product of 
$n_f$ copies of $E^f_w$, and 
\[ \begin{split}
M^1_p 
&= \left(\prod_{f\in I_1} M^f \right)\otimes \bQ_p
= \prod_{f\in I_1} [E^f]^{n_f}\otimes \bQ_p \\
&= \prod_{f\in I_1} \left[\prod_{w\in \cW(f;p)} E^f_w \right]^{n_f}
= \prod_{f\in I_1} \prod_{w\in \cW(f;p)} M^f_w. 
\end{split}\]
A similar calculation shows that
$M^2_p = \prod_{\{f, f^*\}\subset I_2} \prod_{w\in \cW(f;p)} M^f_w$.
The latter part of the argument is straightforward.
\end{proof}

This lemma shows that $M^1_p$ and $M^2_p$ decompose into factors $M_w^f$, 
which can be seen as $E_w^f$-modules discussed in \S\ref{sec:LocalTheory}. 
We will use the notation 
$\cW_{\rm sp}(f;v), \cW_{\rm ur}(f;v), \cW_{\rm rm}(f;v), \cW_\pm(f;v)
\subset \cW(f;v)$ as in \eqref{eq:spurrm}.
Note that $\cW(f;v) = \cW_{\rm sp}(f;v)$ for any $f\in I_2$.

\subsection{Necessity}\label{ss:necessity}
We prove the necessity (i) $\Rightarrow$ (ii) of Theorem \ref{th:LGP}. 
Let $(\Lambda, b)$ be an even unimodular lattice of signature $(r,s)$
having a semisimple $(F, \iota)$-isometry $t$. 
We identify $\Lambda\otimes \bQ$ with the algebra $M$ regarding $t$ as $\alpha$. 
Then, a family $\{b_v\}_{v\in\cV}$ of inner products on $M_v$ can be defined naturally by 
$b_v := b\otimes\bQ_v$ for each $v\in \cV$. 
It is clear that each $b_v$ has the properties \eqref{eq:P1} and \eqref{eq:P2}. 

We then verify the property \eqref{eq:P3}. 
Let $b^\pm$ denote the restriction of $b$ to $M^\pm$. 
Notice that the inner product $b^\pm\otimes\bQ_v$ on $M_v^\pm$ is the same as 
$b_v^{X\mp1} = b_v|_{M^\pm_v\times M^\pm_v}$ for any $v\in\cV$. 
We write $b^\pm_v$ for this inner product. 

\begin{lemma}
We have $\det b^\pm = D_\pm$ in $\bQ^\times/\bQ^{\times 2}$, where 
\begin{equation}\label{eq:D_pm}
D_\pm := \begin{cases}
(-1)^{(n_\pm - \iota(X\mp1))/2} \cdot |F_{12}(\pm 1)|  & \text{if $n_+$ is even} \\
(-1)^{(n_\pm - \iota(X\mp1))/2} \cdot 2 \, |F_{12}(\pm 1)| & \text{if $n_+$ is odd.}
\end{cases}
\end{equation}
\end{lemma}
\begin{proof}
Since we have $\det b^\pm = (-1)^{(n_\pm - \iota(X\mp1))/2} \in \bR^\times/\bR^{\times 2}$, 
it is sufficient to show that
$v_p(\det b^\pm) \equiv v_p(F_{12}(\pm1))$ mod $2$ for each odd prime $p$ and
\[ 
v_2(\det b^\pm) \equiv \begin{cases}
  v_2(F_{12}(\pm 1)) & \text{if $n_+$ is even} \\
  1 + v_2(F_{12}(\pm 1)) & \text{if $n_+$ is odd} 
\end{cases}
\]
mod $2$. Note that $v_p(\det b^\pm) \equiv \dim \partial[M_p^\pm, b_p^\pm, t] \bmod 2$. 
Let $p$ be an odd prime. 
Since $\partial[M, b, t]$ is the trivial class in $W_\Gamma(\bF_p)$, 
so is its image under the projection $W_\Gamma(\bF_p)\to W_\Gamma(\bF_p, \pm1)$. 
Thus 
\[\begin{split}
v_p(\det b^\pm) 
\equiv \dim \partial[M_p^\pm, b_p^\pm, t] 
\equiv \dim \partial\left[\bigoplus_{f\in I_1}\bigoplus_{w\in \cW_{\pm}(f;p)}
(M_w^f, b|_{M_w^f\times M_w^f}, t)  \right]
\equiv v_p(F_{12}(\pm1))
\end{split}\]
mod $2$ by Lemma \ref{lem:dimpartial}.
Let $p = 2$. It follows from \eqref{eq:spinorvaluation} that
\begin{equation}\label{eq:valdetb-}
v_2(\det b^-) 
\equiv v_2(\sn(t)) + v_2(F_{12}(-1))
\equiv \begin{cases}
  v_2(F_{12}(-1)) & \text{if $n_+$ is even} \\
  1 + v_2(F_{12}(-1)) & \text{if $n_+$ is odd}
\end{cases}
\end{equation}
mod $2$. 
If $b^1$ and $b^2$ denote the restrictions of $b$ to $M^1$ and $M^2$ respectively, 
then $\det(b^+ \oplus b^-) = \det(b) \det (b^1\oplus b^2) = \det (b^1\oplus b^2)$ 
in $\bQ^\times/\bQ^{\times 2}$. 
This implies that
\[ v_2(\det b^+) + v_2(\det b^-) 
\equiv v_2(F_{12}(1))+ v_2(F_{12}(-1)) \mod 2
\]
by Lemma \ref{lem:det_iso_formula}. Combining this and \eqref{eq:valdetb-} yields 
\[
v_2(\det b^+) 
\equiv \begin{cases}
  v_2(F_{12}(1)) & \text{if $n_+$ is even} \\
  1 + v_2(F_{12}(1)) & \text{if $n_+$ is odd}
\end{cases}\]
mod $2$. Thus the proof is complete. 
\end{proof}

This lemma shows that $\det b^\pm_v = D_\pm$ in $\bQ_v^\times/\bQ_v^{\times 2}$
for each $v\in \cV$, which is nothing but \eqref{eq:P3}.
Let $b^f$ denote the restriction of $b$ to $M^f$ for each $f\in I$. 
Then $b_v^f = b^f\otimes\bQ_v$, which implies that 
almost all $\epsilon_v(b_v^f)$ equal $0$ and 
$\sum_{v\in \cV} \epsilon_v(b_v^f) = 0$ for each $f\in I$. 
Therefore $\{ b_v \}_v$ is the required family, and the proof of 
(i) $\Rightarrow$ (ii) is complete.

\subsection{Twisting groups as Brauer groups}
In order to prove the implication (ii) $\Rightarrow$ (i) of Theorem \ref{th:LGP}, 
we relate twisting groups to Brauer groups.
In this subsection, we work with a somewhat more general setting.
Let $K$ be a field of characteristic $\neq 2$, and $E$ an extension field of $K$
with a nontrivial involution $\sigma$.
We denote by $(\lambda, \sigma)\in \Br(E^\sigma)$ the Brauer class of 
the cyclic algebra defined from $\lambda\in (E^\sigma)^\times$ and the
cyclic extension $E/E^\sigma$ (see e.g. \cite[Chapter 8, \S 12]{Sc85}).
The map $(E^\sigma)^\times \to \Br(E^\sigma)$ defined by
$\lambda\mapsto (\lambda, \sigma)$
induces the following exact sequence:
\[ 1 \to \Tw(E, \sigma)
\stackrel{(\cdot, \sigma)}{\longrightarrow} \Br(E^\sigma)
\stackrel{\Res_{E/E^\sigma}}{\longrightarrow} \Br(E) 
\]
where $\Res_{E/E^\sigma}$ is the restriction map.

If $K$ is a global field, we denote by $\cW$ the set of all places of $E^\sigma$.
For each $w\in \cW$, the inclusion $E^\times \hookrightarrow E_w^\times$ induces
a map $\Tw(E, \sigma) \to \Tw(E_w, \sigma)$. Let $w\in \cW$.
If $E_w = E\otimes_{E^\sigma}E^\sigma_w$ is a field, 
then $\Tw(E_w, \sigma)$ is of order $2$, and in particular, there exists a unique 
isomorphism $\theta_w: \Tw(E_w, \sigma)\to \bZ/2\bZ$. 
If $E_w \cong E^\sigma_w \times E^\sigma_w$
then $\Tw(E_w, \sigma)$ is a trivial group. In this case 
$\theta_w:\Tw(E_w, \sigma) \to \bZ/2\bZ$ denotes the trivial map. 

\begin{proposition}\label{prop:exsqofTw}
Assume that $K$ is a global field. The sequence
\[1\to \Tw(E, \sigma)
\stackrel{}{\longrightarrow} \bigoplus_{w\in \cW}\Tw(E_w, \sigma)
\stackrel{\sum\theta_w}{\longrightarrow} \bZ/2\bZ \]
is exact.
\end{proposition}
\begin{proof}
This follows from the exact sequence
\[ 0\to \Br(L)\to \bigoplus_{w:\text{place}}\Br(L_w)\stackrel{\sum\inv}{\longrightarrow} 
\bQ/\bZ \to 0  \]
for $L = E$ and $E^\sigma$, see \cite[Theorem 5.7]{BT20}.
\end{proof}

We also need the following formula for Hasse-Witt invariants.

\begin{lemma}\label{lem:CorFormula}
Assume that $K$ is a local field, and let $M$ be a finitely generated free $E$-module. 
Then
\[ \epsilon((M, b[\lambda])) = \epsilon((M,b[1])) 
+ \Cor_{E^\sigma/K}(\lambda, \sigma), \]
where $\Cor_{E^\sigma/K}:\Br(E^\sigma) \to \Br(K)$ is the corestrection map.
\end{lemma}
\begin{proof}
This is a special case of \cite[Theorem 4.3]{BCM03} if $M$ is of rank $1$. 
The result in the general case follows from that in the rank one case.
\end{proof}

\subsection{Sufficiency}
To obtain Theorem \ref{th:LGP}, we prove the sufficiency (ii) $\Rightarrow$ (i).
\smallskip\\
\textit{Proof of Theorem \textup{\ref{th:LGP}}.}
The implication (i) $\Rightarrow$ (ii) is proved in \S\ref{ss:necessity}.   
We show \textup{(ii)} $\Rightarrow$ \textup{(i)}.
Let $\{ b_v \}_{v\in\cV}$ be the family mentioned in \textup{(ii)}.
We write $b_v^{\pm} = b_v^{X\mp1}$ for short. 
Since $\det b_v^\pm = D_\pm$ for each $v\in \cV$ and 
$\sum_{v\in \cV}\epsilon_v(b_v^\pm) = 0$, there exists an inner product
$B^\pm$ on $M^\pm$ such that $B_v^\pm \cong b_v^\pm$ for each $v\in \cV$ 
by \cite[Chapter IV, Proposition 7]{Se73}. Let $f\in I_1$ and $v\in \cV$. 
For each $w\in \cW(f;v)$ the restriction of $b_v$ to $M_w^f$ is denoted by $b_w^f$. 
Because the automorphism $\alpha:M_w^f\to M_w^f$ is 
an isometry with respect to $b_w^f$, we can write 
$b_w^f = b[\lambda_w^f]$ for some $\lambda_w^f\in \Tw(E_w^f, \sigma)$
for each $w\in \cW(f;v)$.
By Lemmas \ref{lem:CorFormula} and \ref{lem:HWsumFormula}, we have 
\[ \sum_{w\in \cW(f;v)}\Cor_{E_w^{f,\sigma}/\bQ_v}(\lambda_w^f, \sigma)
= \sum_{w\in \cW(f;v)}(\epsilon_v(b[\lambda_w^f]) + \epsilon_v(b[1])) 
= \epsilon_v(b_v^f) + \epsilon_v(b[1]).
\]
Summing over $v\in\cV$ yields
$\sum_{v\in \cV}\sum_{w\in \cW(f;v)}
\Cor_{E_w^{f,\sigma}/\bQ_v}(\lambda_w^f, \sigma) = 0$.
By combining this with the commutative diagram
\[ \xymatrix{
\Tw(E_w^f, \sigma) \ar[r]^{(\cdot, \sigma)} \ar[d]^{\theta_w} & 
\Br(E_w^{f,\sigma}) \ar[rr]^-{\Cor_{E_w^{f,\sigma}/\bQ_v}} \ar[d]^\inv &&
\Br(\bQ_v) \ar[d]^\inv \\
\bZ/2\bZ \ar[r]^{\times \frac{1}{2}} & 
\bQ/\bZ \ar[rr]^\id &&
\bQ/\bZ
} \]
we get  
$\sum_{v\in \cV}\sum_{w\in \cW(f;v)}\theta_w(\lambda_w^f) = 0$. 
Thus, by Proposition \ref{prop:exsqofTw}, 
there exists $\lambda^f\in \Tw(E^f,\sigma)$ such that
its image under $\Tw(E^f, \sigma)\to \Tw(E_w^f,\sigma)$
equals $\lambda_w^f$ for any place $w$ of $E^{f,\sigma}$.
We take a hermitian form $h^f:M^f\times M^f\to E^f$ satisfying
$\det h^f = \lambda^f$ and $\idx h^f = \iota(f)/2$.
This is possible by \cite[Theorem 10.6.9]{Sc85}.

Now we define an inner product $B$ on $M$ to be
$B^+ \oplus B^- \oplus B^1 \oplus B^2$, 
where $B^1 = \bigoplus_{f\in I_1}\Tr_{E^f/\bQ}\circ h^f$
and $B^2(x,y) = \Tr_{M^2/\bQ}(x\sigma(y))$.
Since $B_v\cong b_v$ for each $v\in \cV$, 
the property \eqref{eq:P2} implies that there exists an $\alpha$-stable
even unimodular $\bZ_p$-lattice $\Lambda_p$ in $(M_p, B_p)$ for each prime $p$.
We may assume that $\Lambda_p$ coincides with the image of the direct sum 
of the ring of integers in $M$ under $M\to M_p$
for almost all $p$. Then 
\[\Lambda:= \{x\in M \mid \text{the image of $x$ under $M\to M_p$ 
belongs to $\Lambda_p$ for each prime $p$} \}\]
is an $\alpha$-stable even unimodular $\bZ$-lattice in $(M,B)$.
Since $\alpha$ is a semisimple $(F, \iota)$-isometry, the proof is complete. 
\qed

\section{Local-global obstruction}\label{sec:LGO}
Let $F\in \bZ[X], r, s\in \bZ_{\geq 0}$ and $\iota \in \Idx_{r,s}(F)$ be the same
as in the previous section, and let $\cB$ denote the set of families 
$\{b_v\}_{v\in \cV}$ of inner products on $M_v$ such that each $b_v$ 
has the properties \eqref{eq:P1}--\eqref{eq:P3} and that
$\#\{ (v,f) \in \cV\times I \mid \epsilon_v(b_v^f) = 1 \}$ is finite. 

\begin{proposition}
If $F$ satisfies the condition \eqref{eq:Sqcd} then $\cB$ is not empty.
\end{proposition}
\begin{proof}
We can take an inner product $b_\infty$ on $M_\infty$ satisfying
\eqref{eq:P1}--\eqref{eq:P3} by Proposition \ref{prop:realize_idx}. 
Let $p$ be a prime. By the assumption \eqref{eq:Sqcd} we have 
$v_p(F(1))\equiv v_p(F(-1))\equiv 0 \bmod 2$ and $(-1)^nF(1)F(-1) = 1$ mod squares
if $F(1)F(-1) \neq 0$. Hence there exists an inner product $b_p$ on $M_p$
satisfying \eqref{eq:P1}--\eqref{eq:P3} by Proposition \ref{prop:detailedneq2} 
or Theorem \ref{th:detailedQ2}.
Moreover, if 
\begin{equation}\label{eq:*unr}
\text{$p\neq 2$ and $p$ is unramified in $E^f/\bQ$}, \tag{$*$}  
\end{equation}
then we may assume that $(M_p^f, b_p^f)$ contains an $\alpha$-stable (even) 
unimodular lattice for each $f\in I_1\cup I_2$.
In this case we have $\epsilon_p(b_p^f) = 0$ for all $f\in I_1\cup I_2$
because any unimodular $\bZ_p$-lattice ($p\neq 2$) has the trivial Hasse-Witt 
invariant. 
Furthermore, because the image of 
$\partial[M_p^\pm,b_p^\pm, \pm1] = \partial[M_p, b_p, \alpha]$ 
under the projection $W_\Gamma(\bF_p) \to W_\Gamma(\bF_p;\pm1)$
equals $0$, the space $(M_p^\pm, b_p^\pm)$ 
contains a unimodular $\bZ_p$-lattice and
we have $\epsilon_p(b_p^\pm) = 0$ again.
 
Let $\{b_v\}_v$ be a family chosen as above. Almost all primes 
satisfy \eqref{eq:*unr} and thus
$\#\{ (v,f) \in\allowbreak \cV\times I \mid \epsilon_v(b_v^f) = 1 \} < \infty$. 
\end{proof}

In the following we assume that $F$ satisfies the condition \eqref{eq:Sqcd}.

\subsection{Obstruction group and obstruction map}\label{ss:OGandOM}
Set $C(I) := \{ \gamma: I\to \bZ/2\bZ \} = (\bZ/2\bZ)^{\oplus I}$ where 
$I = I_0\cup I_1$, and define a map $\eta : \cB \to C(I)$ by
\[ \{b_v\}_v \mapsto (f \mapsto \sum_{v\in \cV} \epsilon_v(b_v^f)). \]
Theorem \ref{th:LGP} means that 
there exists an even unimodular lattice of signature $(r, s)$ having 
a semisimple $(F, \iota)$-isometry 
if and only if there exists a family $\beta = \{b_v\}_v$ 
such that $\eta(\beta) = \bm{0}$.
We define a group and map, which will be called
the \textit{obstruction group} and \textit{obstruction map}, 
to describe when such a family exists.
We remark that $\gamma \in C(I)$ is the zero map if and only if 
$\gamma \cdot c = 0$ for all $c\in C(I)$, where 
$\gamma \cdot c := \sum_{f\in I} \gamma(f) c(f)$.

First, for two distinct elements $f, g \in I$ we define a set $\Pi_{f,g}$ of prime 
numbers as follows:
\begin{itemize}
\item For $f,g \in I_1$, a prime number $p$ belongs to $\Pi_{f,g}$ if and only if
\begin{itemize}
\item there exist irreducible $*$-symmetric factors $\phi$ and $\psi$ of 
$f$ and $g$ in $\bZ_p[X]$ respectively which are divisible in $\bF_p[X]$ by 
a common irreducible, $*$-symmetric polynomial. 
\end{itemize}
\item For $f\in I$, a prime number $p$ belongs to $\Pi_{f,X\mp1} = \Pi_{X\mp1, f}$ 
if and only if
\begin{itemize}
\item $n_\pm \geq 3$, or $n_\pm = 2$ and $D_\pm \neq -1$ in 
$\bQ_p^\times/\bQ_p^{\times 2}$ where $D_\pm$ is defined in \eqref{eq:D_pm}; and
\item there exists an irreducible $*$-symmetric factor $\phi$ of $f$ 
in $\bZ_p[X]$ which is divisible by $X\mp1$ in $\bF_p[X]$.
\end{itemize}
\end{itemize}
In particular, if $n_\pm = 1$
then $\Pi_{f,X\mp 1} = \Pi_{X\mp 1, f} = \emptyset$ for all $f\in I$. 
If $n_\pm = 2$ then $\Pi_{f,X\mp 1}$ depends on $\iota$ because so does $D_\pm$. 

Then, we define an equivalence relation $\sim$ on $C(I)$ as the one generated by 
the following relation $R$:
\[ R(\gamma, \gamma') \iff \gamma' = \gamma + \bm{1}_{\{f, g\}}
\text{ for some $f,g \in I$ such that $\Pi_{f,g}$ is not empty}. \]
Here $\bm{1}_H$ is the characteristic function of $H\subset I$.

\begin{theorem}\label{th:equivalenceclass}
The image $\im \eta \subset C(I)$ of $\eta$ coincides with an equivalence class 
with respect to $\sim$.
\end{theorem}

This theorem will be proved in the next subsection. Set 
$C_\sim(I) := 
\{ c\in C(I) \mid c(f) =\allowbreak c(g)\text{ if } \Pi_{f,g}\neq \emptyset \}$.
Notice that if $\gamma \sim \gamma'$ then $\gamma\cdot c = \gamma'\cdot c$ 
for any $c\in C_\sim(I)$. Hence, the map 
\begin{equation*}
C_\sim(I) \to \bZ/2\bZ, \; c\mapsto \eta(\beta)\cdot c
\end{equation*}
is defined independently of the choice of $\beta\in\cB$ 
by Theorem \ref{th:equivalenceclass}. 
The following proposition shows that this map factors through the 
\textit{obstruction group} $\Omega := C_\sim(I)/\{\text{constant maps}\}$. 
The induced homomorphism $\Omega \to \bZ/2\bZ$ is called the \textit{obstruction map} 
and denoted by $\ob$. 

\begin{proposition}\label{prop:etabeta_has_even1}
We have $\eta(\beta)\cdot \bm{1}_I = 0$ for any $\beta \in \cB$. 
\end{proposition}
\begin{proof}
Let $B^\pm$ be an inner product on $M^\pm$ whose Gram matrix is
$\diag (D_\pm, 1,\cdots,1)$, and 
$B^f$ an inner product on $M^f$ defined by
\[ B^f(x,y) = \Tr_{M^f/\bQ}(x\sigma(y)) \quad\text{for $x,y \in M^f$}, \]
for each $f\in I_1\cup I_2$. Then, set
$B := B^+\oplus B^- \oplus \bigoplus_{f\in I_1} B^f
\oplus \bigoplus_{\{f, f^*\}\subset I_2} B^f$.
Let $\beta = \{b_v\}_v\in \cB$. By Lemma \ref{lem:HWsumFormula} we have
$ \epsilon_{v}(b_v) - \sum_{f\in I} \epsilon_v(b_v^f) 
= \epsilon_{v}(B) - \sum_{f\in I} \epsilon_v(B^f).$
Summing over $v\in\cV$ yields
\[ \sum_{v\in\cV}\epsilon_{v}(b_v) - \sum_{v\in\cV}\sum_{f\in I} \epsilon_v(b_v^f) 
= \sum_{v\in\cV}\epsilon_{v}(B) - \sum_{v\in\cV}\sum_{f\in I} \epsilon_v(B^f)
= 0 \]
since $B$ and $B^f$ are global objects.
Furthermore, we have $\sum_{v\in\cV}\epsilon_{v}(b_v) = 0$ because
$b_v \cong\allowbreak \Lambda_{r,s}\otimes \bQ_v$ for all $v\in \cV$, where
$\Lambda_{r,s}$ is an even unimodular $\bZ$-lattice of signature $(r,s)$, 
see \cite[\S 10]{BT20}.
Therefore we get 
$\eta(\beta)\cdot \bm{1}_I 
= \sum_{v\in\cV}\sum_{f\in I} \epsilon_v(b_v^f)
=0$.
\end{proof}

If there exists $\beta\in \cB$ with $\eta(\beta) = \bf{0}$ then 
the obstruction map $\ob:\Omega \to \bZ/2\bZ$ is zero.  
To prove Theorem \ref{th:main1}, we also need the converse.  
We start by describing the equivalence class in $C(I)$ containing $\bm{0}$. 
Let $G(F)$ denote the graph such that the vertices are all elements in $I$, and that
two distinct vertices $f,g\in I$ are joined if and only if $\Pi_{f,g} \neq \emptyset$. 
Then $C_\sim(I)$ is the set of maps in $C(I)$ which are constant on each connected 
component of $G(F)$. We remark that the graph $G(F)$ may depend on $\iota$.

\begin{proposition}
The subset $C_\sim(I)^\perp 
:= \{ \gamma\in C(I) \mid \gamma\cdot c = 0 \text{ for any $c\in C_\sim(I)$} \}$
is the equivalence class containing $\bm{0}$ in $C(I)$. 
\end{proposition}
\begin{proof}
Let $\cE$ denote the equivalence class containing $\bm{0}$.  
We start by showing the inclusion $\cE \subset C_\sim(I)^\perp$. 
Let $f,g\in I$ be distinct elements with $\Pi_{f,g}\neq \emptyset$. 
For any $c\in C_\sim(I)$ we have 
\[ \bm{1}_{\{f, g\}}\cdot c
= c(f) + c(g)
= 2c(f)
=0, \] 
which implies that $\bm{1}_{\{f, g\}} \in C_\sim(I)^\perp$. 
Since any element of $\cE$ can be expressed as a sum of maps $\bm{1}_{\{f, g\}}$
with $\Pi_{f,g}\neq \emptyset$, we have $\cE \subset C_\sim(I)^\perp$. 

Let us prove the reverse inclusion $C_\sim(I)^\perp \subset \cE$.  
We claim that for any nonzero $\gamma\in C_\sim(I)^\perp$ there exists 
$\gamma'\in C_\sim(I)^\perp$ such that $\gamma'\sim\gamma$ and 
$\#\Supp(\gamma') < \#\Supp(\gamma)$, 
where $\Supp(\gamma):=\{ f\in I \mid \gamma(f)\neq 0 \}$. 
To prove this, take an element $f\in \Supp(\gamma)$ and concider the 
connected component $H\subset I$ in $G(F)$ containing $f$. 
Since $\bm{1}_H\in C_\sim(I)$ we have 
$\sum_{h\in H} \gamma(h) = \gamma\cdot\bm{1}_H = 0$. This implies that $H$ has an element 
$g$ such that $g\neq f$ and $\gamma(g)\neq 0$. 
Then there exists a path from $f$ to $g$ 
in $G(F)$ via $h_1, h_2, \ldots, h_k \in H$, where 
$\Pi_{f,h_1}, \Pi_{h_{j},h_{j+1}} \; (j=1,\ldots, k-1)$, and 
$\Pi_{h_k,g}$ are not empty. This means that
\[\gamma' := \gamma 
+ \bm{1}_{\{f, h_1\}}
+ \bm{1}_{\{h_1, h_2\}} + \cdots 
+ \bm{1}_{\{h_{k-1}, h_k\}}
+ \bm{1}_{\{h_k, g\}} \in C_\sim(I)^\perp
\]
is equivalent to $\gamma$. Since $\Supp(\gamma') = \Supp(\gamma)\setminus\{f,g\}$, 
we get $\#\Supp(\gamma') = \#\Supp(\gamma)-2 <\#\Supp(\gamma)$. 
This completes the proof of the claim. 

Let $\gamma\in C_\sim(I)^\perp$. Then we can see that $\gamma\sim\bm{0}$ 
by repeatedly applying the claim proved above. 
Thorefore we obtain $C_\sim(I)^\perp \subset \cE$. 
\end{proof}

This proposition shows that the obstruction map vanishes only if the image 
$\im \eta\subset C(I)$ coincides with the equivalence class containing $\bm{0}$, 
or equivalently, there exists a family $\beta\in\cB$ with $\eta(\beta) = \bm{0}$. 
\smallskip
\\
\textit{Proof of Theorem \textup{\ref{th:main1}}.}
To summarize the discussion so far, we obtain Theorem \ref{th:main1}. 
\qed

\subsection{Image of the map $\eta:\cB\to C(I)$}
The purpose of this subsection is to prove Theorem \ref{th:equivalenceclass}. 
We keep the notation of the previous subsection.
For $f\in I_1$, a prime $p$ and $w\in \cW(f;p)\setminus\cW_{\rm sp}(f;p)$, 
the symbol $f_w$ will denote the irreducible factor of $f$ in $\bZ_p[X]$
corresponding to the place $w$. 
Recall the classification theorem of inner products over the $p$-adic field
$\bQ_p$: Two inner products over $\bQ_p$ are isomorphic if and only if 
they have the same dimension, same determinant, and same Hasse-Witt invariant, 
see \cite[Chapter IV, Theorem 7]{Se73}. 
We begin with the following lemma. 
For a prime $p$ and an inner product $b_p^\pm$ on $M_p^\pm$, we write
$\partial[b_p^\pm] = \partial[M_p^\pm, b_p^\pm, \pm 1]$ for short. 

\begin{lemma}\label{lem:partial_inv}
Let $p$ be a prime.
\begin{enumerate}
\item Let $p\neq 2$, and let $b_p^\pm, \widehat{b}_p^\pm$ be inner products on 
$M_p^\pm$ having the same determinant. Assume that $n_\pm >1$. 
Then $\partial[b_p^\pm] = \partial[\widehat{b}_p^\pm]$ if and only if
$\epsilon_p(b_p^\pm) = \epsilon_p(\widehat{b}_p^\pm)$.
\item Let $f\in I_1, w\in \cW(f;p)$, and 
$\lambda_w^f, \widehat{\lambda}_w^f\in \Tw(E_w^f, \sigma)$.
If $w\notin \cW_{\rm rm}(f;2)$ then the following are equivalent:
\begin{itemize}
\item[\textup{(a)}] $\lambda_w^f = \widehat{\lambda}_w^f$.
\item[\textup{(b)}] $\epsilon_p(b[\lambda_w^f]) = \epsilon_p(b[\widehat{\lambda}_w^f])$.
\item[\textup{(c)}] $\partial[M_w^f, b[\lambda_w^f],\alpha_w^f] 
= \partial[M_w^f, b[\widehat{\lambda}_w^f],\alpha_w^f]$. 
\end{itemize} 
Here $\alpha_w^f$ denotes the image of $X$ in $M_w^f$ or $E_w^f$.
\end{enumerate}
\end{lemma}
\begin{proof}
\textup{(i).} 
This follows by directly computing the value of the map 
$\partial:W_\Gamma(\bQ_p, \pm1)\to W_\Gamma(\bF_p, \pm1)$ for a representative 
of each isomorphism class for the inner products on $M_p^\pm$.

\textup{(ii).} The implication (a) $\Rightarrow$ (b) is obvious. If (b) holds, 
then $b[\lambda_w^f]$ and $b[\widehat{\lambda}_w^f]$ are isomorphic because they have
the same dimension, same determinant, and same Hasse-Witt invariant. In particular, 
the condition \textup{(c)} holds.
If $w\notin \cW_{\rm rm}(f;2)$, Theorem \ref{th:imtheofTw_to_W} implies
that $\partial_{M_w^f, \alpha_w^f}$ is injective. Therefore, we have
\textup{(c) $\Rightarrow$ (a)}.
\end{proof}

\begin{proposition}\label{prop:im>EC}
For any $\beta\in \cB$ and $f, g \in I$ with $\Pi_{f,g} \neq \emptyset$, 
there exists $\widehat{\beta}\in \cB$ such that 
$\eta(\beta) + \bm{1}_{\{f, g\}} = \eta(\widehat{\beta})$.
In particular, the image $\im \eta\subset C(I)$ contains an equivalence class.
\end{proposition}
\begin{proof}
Let $\beta = \{b_v\}_v\in \cB$. Take $f,g\in I$ with 
$\Pi_{f,g} \neq \emptyset$, and let $p\in \Pi_{f,g}$. Assume that $f,g\in I_1$. 
By the definition of $\Pi_{f,g}$, there exist places 
$w_0\in \cW(f;p)\setminus\cW_{\rm sp}(f;p)$ and 
$u_0\in \cW(g;p)\setminus\cW_{\rm sp}(g;p)$ such that 
$f_{w_0}$ mod $p$ and $g_{u_0}$ mod $p$ have a common irreducible, $*$-symmetric 
factor $\bar{h}$ in $\bF_p[X]$.
Note that $\bar{h}$ is the minimal polynomial of $\bar\alpha_{w_0}^f$
and $\bar\alpha_{u_0}^g$, which implies that there exists an irreducible
representation $\chi$ of $\Gamma$ over $\bF_p$ such that 
$\partial[M_{w_0}^f, b_{w_0}^f, \alpha_{w_0}^f]$ and 
$\partial[M_{u_0}^g, b_{u_0}^g, \alpha_{u_0}^g]$ are in $W_\Gamma(\bF_p; \chi)$. 

According to the decompositions $M_p^f = \bigoplus_{w\in \cW(f,p)} M_w^f$ and 
$M_p^g = \bigoplus_{u\in \cW(g,p)} M_u^g$, we can write 
\begin{alignat*}{2}
b_p^f &= \bigoplus_{w\in \cW(f,p)} b_w^f, \quad &b_w^f = b[\lambda_w^f] 
\quad&\text{for some $\lambda_w^f\in \Tw(E_w^f, \sigma)$}, \\ 
b_p^g &= \bigoplus_{u\in \cW(g,p)} b_u^g, \quad &b_u^g = b[\lambda_u^g] 
\quad&\text{for some $\lambda_u^g\in \Tw(E_u^g, \sigma)$}
\end{alignat*} 
by Lemma \ref{lem:tr_hermitian}. 
Let $\widehat\lambda_{w_0}^f \in \Tw(E_w^f, \sigma)$ and 
$\widehat\lambda_{u_0}^g \in \Tw(E_u^g, \sigma)$ be elements satisfying
$\widehat\lambda_{w_0}^f \neq\allowbreak \lambda_{w_0}^f$ and 
$\widehat\lambda_{u_0}^g\neq \lambda_{u_0}^g$, and set
\begin{equation}\label{eq:hat_p}
\widehat{b}_p^f := b[\widehat\lambda_{w_0}^f]\oplus \bigoplus_{w\neq w_0} b_w^f, \quad
\widehat{b}_p^g := b[\widehat\lambda_{u_0}^g]\oplus \bigoplus_{u\neq u_0} b_u^g, 
\end{equation}
\[\begin{split}
&\widehat{b}_p := b_p^0 \oplus \widehat{b}_p^f \oplus \widehat{b}_p^g \oplus 
\left(\bigoplus_{k\in I_1\setminus \{f,g\}} b_p^k\right) \oplus b_p^2, \quad\text{and}\\
&\widehat\beta := \{\widehat{b}_v\}_v \quad\text{where $\widehat{b}_v = b_v$ for $v\neq p$}.
\end{split}
\]
Here $b_p^0$ and $b_p^2$ are the restrictions of $b_p$ to 
$M_p^0$ and $M_p^2$ respectively.

\textit{\textup{Claim 1:} We have $\widehat\beta\in \cB$.}
Since $\widehat{b}_v = b_v$ for any $v\neq p$, 
each $\widehat{b}_v$ ($v\neq p$) has the properties \eqref{eq:P1}--\eqref{eq:P3}, 
and $\#\{ (v,f) \in \cV\times I \mid \epsilon_v(\widehat{b}_v^f) = 1 \}$ is finite. 
Furthermore, \eqref{eq:P1} and \eqref{eq:P3} are obvious for $\widehat{b}_p$.  
Therefore, it suffices to show that $\widehat{b}_p$ has the property \eqref{eq:P2}.
By using Theorem \ref{th:imtheofTw_to_W} and Lemma \ref{lem:partial_inv}, 
some calculations show that 
\[ 
\partial[M_{w_0}^f, b[\widehat\lambda_{w_0}^f], \alpha_{w_0}^f]
+ \partial[M_{u_0}^g, b[\widehat\lambda_{u_0}^g], \alpha_{u_0}^g]
- \partial[M_{w_0}^f, b[\lambda_{w_0}^f], \alpha_{w_0}^f]
- \partial[M_{u_0}^g, b[\lambda_{u_0}^g], \alpha_{u_0}^g]
= 0. 
\]
Then 
\[ \begin{split}
\partial[M_p, \widehat{b}_p, \alpha]
&= \partial[M_p, b_p, \alpha] 
- \partial[M_{w_0}^f, b[\lambda_{w_0}^f], \alpha_{w_0}^f]
- \partial[M_{u_0}^g, b[\lambda_{u_0}^g], \alpha_{u_0}^g] \\
& \quad
+ \partial[M_{w_0}^f, b[\widehat\lambda_{w_0}^f], \alpha_{w_0}^f]
+ \partial[M_{u_0}^g, b[\widehat\lambda_{u_0}^g], \alpha_{u_0}^g]\\
&= 0, 
\end{split}
\]
which implies that $(M_p, \widehat{b}_p)$ contains an $\alpha$-stable unimodular lattice
by Theorem \ref{th:ULvanish}. 
If $p$ is odd then the lattice is even and we are done.  
Let $p=2$. We use Proposition \ref{prop:NScdforZ2EUL}. 
The condition (i) of Proposition \ref{prop:NScdforZ2EUL} has already been proved. 
The condition (iii) follows from the formula \eqref{eq:spinorvaluation}. 
For the condition (ii), it is sufficient to show that $\widehat{b}_2\cong b_2$. 
We have
\[ \epsilon_2(\widehat{b}_2) - \epsilon_2(b_2)
= \epsilon_2(b[\widehat\lambda_{w_0}^f]) + \epsilon_2(b[\widehat\lambda_{u_0}^g])
- \epsilon_2(b[\lambda_{w_0}^f]) - \epsilon_2(b[\lambda_{u_0}^g])
= 0
\]
by Lemmas \ref{lem:HWsumFormula} and \ref{lem:partial_inv} 
(even if $w_0\in \cW_{\rm rm}(f;2)$ or $u_0\in \cW_{\rm rm}(g;2)$). 
Hence $\widehat{b}_2$ and $b_2$ have 
the same dimension, same determinant, and same Hasse-Witt invariant, 
which means that $\widehat{b}_2\cong b_2$ as required. 
Therefore $(M_2, \widehat{b}_2)$ contains an $\alpha$-stable even unimodular lattice, 
that is, $\widehat{b}_2$ has the property \eqref{eq:P2}. 
This completes the proof of Claim 1.  
\smallskip

\textit{\textup{Claim 2:} We have $\eta(\beta) + \bm{1}_{\{f,g\}} = \eta(\widehat\beta)$.}
It is obvious that 
$\eta(\beta)(k) = \eta(\widehat\beta)(k)$ for $k\neq f,g$, and we have
\[ \eta(\widehat\beta)(f) - \eta(\widehat\beta)(f) 
= \epsilon_p(\widehat{b}_p^f) - \epsilon_p(b_p^f)
= \epsilon_p(b[\widehat\lambda_{w_0}^f]) - \epsilon_p(b[\lambda_{w_0}^f])
= 1.
\]
The same calculation yields $\eta(\widehat\beta)(g) - \eta(\widehat\beta)(g) = 1$, 
and thus we arrive at the claim.
\smallskip

The proof for the case $f,g \in I_1$ has been completed now. 
Let $f(X) = X\mp 1$. In this case, define an inner product 
$\widehat{b}_p^f = \widehat{b}_p^\pm$ to satisfy
$ \det \widehat{b}_p^\pm = D_\pm$ and 
$\epsilon_p(\widehat{b}_p^\pm)\neq \epsilon_p(b_p^\pm)$, and set
\begin{equation}\label{eq:tildebeta_0}
\widehat{b}_p := \widehat{b}_p^f \oplus \widehat{b}_p^g \oplus 
\left(\bigoplus_{k\in I\setminus \{f,g\}} b_p^k\right) \oplus b_p^2 \quad\text{and}\quad
\widehat\beta := \{\widehat{b}_v\}_v \quad\text{where $\widehat{b}_v = b_v$ for $v\neq p$}.
\end{equation}
Here $\widehat{b}_p^g$ is defined as in \eqref{eq:hat_p} if $g\in I_1$ and  
as above if $g\in I_0$. 
Then the above two claims hold similarly and the proof is complete. 
\end{proof}

\begin{proposition}\label{prop:im<EC}
The image $\im \eta \subset C(I)$ is contained in an equivalence class.
In other words $\eta(\beta)\sim \eta(\widehat\beta)$ for any $\beta, \widehat\beta \in \cB$. 
\end{proposition}
\begin{proof}
Let $\beta=\{b_v\}_v, \widehat\beta = \{\widehat{b}_v\}_v\in \cB$, and set
$V(\beta, \widehat\beta) := \{ v\in \cV \mid \eta_v(\beta) \neq \eta_v(\widehat\beta) \}$, 
where $\eta_v(\beta) = (f \mapsto \epsilon_v(b_v^f)) \in C(I)$.
Note that $\eta(\beta)(f) = \sum_{v\in \cV} \eta_v(\beta)(f)$
for any $\beta\in \cB$ and $f\in I$. 
Since almost all $\epsilon_v(b_v^f)$ and $\epsilon_v(\widehat{b}_v^f)$ are trivial and 
$b_\infty \cong \widehat{b}_\infty$, the set $V(\beta, \widehat\beta)$ is a finite set of primes.
If $\#V(\beta, \widehat\beta) = 0$, then $\eta(\beta) = \eta(\widehat\beta)$ and
there is nothing to prove.
By induction on $\#V(\beta, \widehat\beta)$, 
it is sufficient to show that there exists $\widetilde\beta\in \cB$ 
such that $\eta(\widetilde\beta) \sim \eta(\beta)$ and 
$V(\widetilde\beta, \widehat\beta) \subsetneq V(\beta, \widehat\beta)$.
Let $p\in V(\beta, \widehat\beta)$.

\textit{\textup{Claim 1:} There exists $\widetilde\beta = \{\widetilde{b}_v\}_v\in \cB$ 
such that $\eta(\widetilde\beta) \sim \eta(\beta),
V(\widetilde\beta, \widehat\beta) \subset V(\beta, \widehat\beta), \allowbreak
\partial[\widetilde{b}_p^+] = \partial[\widehat{b}_p^+]$, and 
$\partial[\widetilde{b}_p^-] = \partial[\widehat{b}_p^-]$.}
If $p = 2$ then $\partial[b_2^\pm] = \partial[\widehat{b}_2^\pm]$ since 
$\det b_2^\pm = \det \widehat{b}_2^\pm$, and the claim is obvious. 
Let $p\neq 2$ and $\partial[b_p^+] \neq \partial[\widehat{b}_p^+]$.
Note that neither $n_+ = 1$ nor $n_+ = 2$ and $D_+ = -1$ in 
$\bQ_p^\times/\bQ_p^{\times 2}$ occurs.
We have $\partial[M_p, b_p, \alpha] = \partial[M_p, \widehat{b}_p, \alpha] = 0$, and
in particular, the images of them under the projection 
$W_\Gamma(\bF_p)\to W_\Gamma(\bF_p;1)$ are $0$.
Thus there exists $f\in I_1$ and $w_0\in \cW(f;p)$ such that
$\partial[b_{w_0}^f] \neq \partial[\widehat{b}_{w_0}^f]$ in $W_\Gamma(\bF_p; 1)$ 
and that $(X-1) \mid f_{w_0}$ mod $p$, where we write 
$\partial[b_{w_0}^f] = \partial[M_{w_0}^f, b_{w_0}^f, \alpha_{w_0}^f]$ simply.
Set
\[\begin{split}
&\widetilde{b}_p := \widehat{b}_p^+ \oplus b_p^- \oplus
\left( \widehat{b}_{w_0}^f \oplus \bigoplus_{w\in \cW(f:p)\setminus\{w_0\}} b_w^f \right)
\oplus \bigoplus_{k\in I_1\setminus \{f\}} b_p^k 
\oplus b_p^2
\quad\text{and}\\
&\widetilde\beta = \{\widetilde\beta_v\}_v 
\quad\text{where $\widetilde{b}_v = b_v$ for $v\neq p$}.
\end{split}\]
Then $\widetilde\beta$ belongs to $\cB$ as in 
Claim 1 of the proof of Proposition \ref{prop:im>EC}.
Furthermore
\[\begin{split}
&\eta(\widetilde\beta)(X-1) - \eta(\beta)(X-1) 
= \epsilon_p(\widehat{b}_p^+) - \epsilon_p(b_p^+)
= 1, \\
&\eta(\widetilde\beta)(f) - \eta(\beta)(f) 
= \epsilon_p(\widehat{b}_p^f) - \epsilon_p(b_p^f)
= 1
\end{split} 
\]
by Lemma \ref{lem:partial_inv}, and 
$\eta(\widetilde\beta)(k) = \eta(\beta)(k)$ for all $k\in I\setminus\{X-1, f\}$.
These equations mean that
$\eta(\widetilde\beta) = \eta(\beta) + \bm{1}_{\{X-1, f\}}$. 
Since $\Pi_{X-1,f}$ contains the prime $p$ and is non-empty, 
we get $\eta(\widetilde\beta)\sim \eta(\beta)$. 
It is clear that
$V(\widetilde\beta, \widehat\beta) \subset V(\beta, \widehat\beta)$ and 
$\partial[\widetilde{b}_p^+] = \partial[\widehat{b}_p^+]$ 
by the definition of $\widetilde\beta$.
If $\partial[\widetilde{b}_p^-] \neq \partial[\widehat{b}_p^-]$ then 
we repeat a similar procedure to obtain 
$\partial[\widetilde{b}_p^-] = \partial[\widehat{b}_p^-]$.
The proof of Claim 1 is complete. 
\smallskip

We assume that 
$\partial[b_p^+] = \partial[\widehat{b}_p^+]$ and
$\partial[b_p^-] = \partial[\widehat{b}_p^-]$ 
without loss of generality by Claim 1. Set
$D_p(\beta, \widehat{\beta}) :=\allowbreak \bigcup_{f\in I_1}\{ w\in \cW(f;p) 
\mid \partial[b_w^f] \neq \partial[\widehat{b}_w^f]\}$.

\textit{\textup{Claim 2:} There exists $\widetilde\beta\in \cB$ such that 
$\eta(\widetilde{\beta}) \sim \eta(\beta), 
V(\widetilde{\beta}, \widehat{\beta}) \subset V(\beta, \widehat{\beta})$ and
$D_p(\widetilde\beta, \widehat\beta) =\allowbreak \emptyset$}.
Use induction on $\# D_p(\beta, \widehat{\beta})$, the case 
$\# D_p(\beta, \widehat{\beta}) = 0$ being obvious. 
Let $\# D_p(\beta, \widehat{\beta}) > 0$, and choose $f\in I_1$ and $w_0\in\cW(f;p)$
satisfying $\partial[b_{w_0}^f] \neq \partial[\widehat{b}_{w_0}^f]$.
There is an irreducible representation $\chi$ of $\Gamma$ such that 
$\partial[b_{w_0}^f], \partial[\widehat{b}_{w_0}^f] \in W_\Gamma(\bF_p;\chi)$. 
Since the images of $\partial[M_p, b_p, \alpha]$ and 
$\partial[M_p, \widehat{b}_p, \alpha]$ under the projection 
$W_\Gamma(\bF_p)\to W_\Gamma(\bF_p;\chi)$ are the trivial class, 
we can take $g\in I_1$ and $u_0\in \cW(g;p)$ satisfying 
$\partial[b_{u_0}^g] \neq \partial[\widehat{b}_{u_0}^g]$ in $W_\Gamma(\bF_p;\chi)$. 
Set 
\begin{equation}\label{eq:tildebeta}
\begin{split}
&\widetilde{b}_p^f := \widehat{b}_{w_0}^f\oplus \bigoplus_{w\neq w_0} b_w^f, \quad
\widetilde{b}_p^g := \widehat{b}_{u_0}^g \oplus \bigoplus_{u\neq u_0} b_u^g,\\  
&\widetilde{b}_p := b_p^0 \oplus \widetilde{b}_p^f \oplus \widetilde{b}_p^g \oplus 
\bigoplus_{k\in I_1\setminus \{f,g\}} b_p^k \oplus b_p^2, \quad\text{and}\\
&\widetilde\beta := \{\widetilde{b}_v\}_v 
\quad\text{where $\widetilde{b}_v = b_v$ for $v\neq p$,}
\end{split}  
\end{equation}
then we have $\widetilde\beta\in \cB$,  
$\eta(\widetilde{\beta}) = \eta(\beta) + \bm{1}_{\{f,g\}} \sim \eta(\beta)$, 
and $V(\widetilde\beta,\widehat{\beta}) \subset V(\beta,\widehat{\beta})$
as in Claim 1. Because
$\# D_p(\widetilde\beta,\widehat{\beta}) = \# D_p(\beta,\widehat{\beta})-2$, 
we arrive at the claim by induction.
\smallskip

Now we take $\widetilde\beta\in \cB$ mentioned in Claim 2.
If $p\neq 2$ then the equation $\eta_p(\widetilde\beta) = \eta_p(\widehat\beta)$
follows from
$\partial[\widetilde{b}_p^+] = \partial[\widehat{b}_p^+]$,
$\partial[\widetilde{b}_p^-] = \partial[\widehat{b}_p^-]$, 
$D_p(\widetilde\beta, \widehat\beta) = \emptyset$, and Lemma \ref{lem:partial_inv}.
Therefore we get $\eta(\widetilde\beta) \sim \eta(\beta)$ and
$V(\widetilde\beta, \widehat\beta) 
\subset V(\beta, \widehat\beta)\setminus\{p\}
\subsetneq\allowbreak V(\beta, \widehat\beta)$ as required. 

Let $p = 2$. If $\eta_2(\widetilde\beta) = \eta_2(\widehat\beta)$ then
$V(\widetilde\beta, \widehat\beta) 
\subset V(\beta, \widehat\beta)\setminus\{2\}
\subsetneq V(\beta, \widehat\beta)$ and we are done. 
If $\eta_2(\widetilde\beta) \neq \eta_2(\widehat\beta)$
then assume that $\beta = \widetilde\beta$ without loss of generality.
Then there exists $f\in I$ such that 
$\eta_2(\beta)(f) \neq \eta_2(\widehat\beta)(f)$. 
Suppose that $f\in I_1$. Then we have
\[ 1 = \epsilon_2(\widehat{b}_2^f) - \epsilon_2(b_2^f)
= \sum_{w\in \cW(f;2)} \epsilon_2(\widehat{b}_w^f)
- \sum_{w\in \cW(f;2)} \epsilon_2(b_w^f), \]
and this implies that there exists $w_0\in \cW(f;2)$ satisfying 
$\epsilon_2(\widehat{b}_{w_0}^f) \neq \epsilon_2(b_{w_0}^f)$. 
Moreover, the place $w_0$ must be in $\cW_{\rm rm}(f;2)$ by Lemma \ref{lem:partial_inv}.
In particular 
$\partial[b_{w_0}^f], \partial[\widehat{b}_{w_0}^f] \in W_\Gamma(\bF_2; 1)$.
On the other hand, the equation
\[ 0 = \epsilon_2(\widehat{b}_2) - \epsilon_2(b_2) 
= \sum_{k\in I} \epsilon_2(\widehat{b}_2^k) 
- \sum_{k\in I} \epsilon_2(b_2^k) \]
implies that there exists $g\in I$ such that 
$\epsilon_2(\widehat{b}_2^g) \neq \epsilon_2(b_2^g)$.
Suppose that $g\in I_1$, then there exists $u_0\in \cW_{\rm rm}(g;2)$ satisfying
$\epsilon_2(\widehat{b}_{u_0}^g) \neq \epsilon_2(b_{u_0}^g)$ as above.
Once again, we define $\widetilde{\beta}$ as in \eqref{eq:tildebeta}.
In the case $f \in I_0$ or $g\in I_0$, we also define
$\widetilde{\beta}$ as in \eqref{eq:tildebeta_0}.
Then $\widetilde\beta \in \cB$ and $\eta(\widetilde\beta) \sim \eta(\beta)$ 
as before. Moreover, because
\[ \{k\in I \mid \eta_2(\widetilde\beta)(k) \neq \eta_2(\widehat\beta)(k)\}
= \{k\in I \mid \eta_2(\beta)(k) \neq \eta_2(\widehat\beta)(k)\} 
\setminus \{ f, g \}, 
\]
we can assume that $\eta_2(\widetilde\beta) = \eta_2(\widehat\beta)$ by induction.
Thus, we obtain $V(\widetilde\beta, \widehat\beta) 
\subset V(\beta, \widehat\beta)\setminus\{2\}
\subsetneq V(\beta, \widehat\beta)$, which completes the proof. 
\end{proof}

\noindent
\textit{Proof of Theorem \textup{\ref{th:equivalenceclass}}.}
Propositions \ref{prop:im>EC} and \ref{prop:im<EC} mean that
the image $\im\eta$ coincides with an equivalence class.
This is Theorem \ref{th:equivalenceclass}. \qed

\subsection{Applications}
As an application of Theorem \ref{th:main1}, 
we begin with the following theorem, which plays a crucial role in proving
Theorem \ref{th:mainK3}.
\begin{theorem}\label{th:mainapp}
Let $r$ and $s$ be non-negative integers with $r\equiv s \bmod 8$, and
$S\in\bZ[X]$ an irreducible $*$-symmetric polynomial of degree 
$r + s - 2$. Assume that $F(X):=\allowbreak(X-1)(X+1)S(X)$ satisfies 
the condition \eqref{eq:Sgcd}.
Then there exists an even unimodular lattice of signature $(r,s)$ having 
a semisimple $(F, \iota)$-isometry for any $\iota \in \Idx_{r,s}(F)$.
\end{theorem}
\begin{proof}
Fix $\iota \in \Idx_{r,s}(F)$ arbitrarily, and let $\beta = \{b_v\}_v \in \cB$.
Note that we have $I = \{X-1, X+1, S\}$.
For each $v\in \cV$, we have $\epsilon_v(b_v^\pm) = 0$
since $b_v^\pm$ is $1$-dimensional. 
This implies that $\eta(\beta)(X-1) = \eta(\beta)(X+1) = 0$, and furthermore  
$\eta(\beta)(S) = \eta(\beta)\cdot\bm{1}_{I} = 0$
by Proposition \ref{prop:etabeta_has_even1}.
Therefore $\eta(\beta) = \bm{0}$, and the obstruction map vanishes. 
This means that
there exists an even unimodular lattice of signature $(r,s)$ having 
a semisimple $(F, \iota)$-isometry by Theorem \ref{th:main1}.
\end{proof}

The rest of this section is devoted to the proof of Theorem \ref{th:mainapp2}.

\begin{lemma}\label{lem:nosymmfactor}
Let $f\in\bZ[X]$ be a $+1$-symmetric polynomial with
$f(1)f(-1)\neq0$, and $p$ a prime.
If $f$ has no irreducible $*$-symmetric factor in $\bZ_p[X]$ which is divisible 
by $X\mp 1$ in $\bF_p[X]$, then $v_p(f(\pm1)) \equiv 0 \bmod 2$, and moreover,
if $p=2$ then $(-1)^{(\deg f)/2}f(1)f(-1) = 1$ or $-3$ in 
$\bQ_2^\times/\bQ_2^{\times2}$. 
\end{lemma}
\begin{proof}
We have a decomposition $f(X) = g(X)h(X)$ in $\bZ_p[X]$ such that  
($g$ mod $p) = (X\mp 1)^{\deg g}$ in $\bF_p[X]$ and $h(\pm1)\neq 0 \bmod p$. 
The assumption of the lemma means that $g$ is of type $2$ (in $\bZ_p[X]$), so 
$g$ is expressed as $g = kk^*$ for some $k\in\bZ_p[X]$. 
In this case we have
\[ v_p(g(\pm1)) 
= v_p(k(\pm1) k^*(\pm1))
= v_p(k(\pm1) k(0)^{-1}(\pm1)^{\deg k} k(\pm1))
\equiv 0 \mod 2\]
and thus 
\[v_p(f(\pm1)) 
= v_p(g(\pm1)) + v_p(h(\pm 1))
\equiv 0 \mod 2. 
\]
Moreover, if $p=2$ then
\[ (-1)^{(\deg g)/2} g(1)g(-1)
= (-1)^{(\deg g)/2} k(1)k^*(1)k(-1)k^*(-1)
= 1 \quad\text{in } \bQ_2^\times/\bQ_2^{\times2}. 
\]
Let $\phi\in\bZ_2[X]$ be the \textit{trace polynomial} of $h$, that is, 
the monic polynomial defined by the equation 
$h(X) =\allowbreak X^{(\deg h)/2} \phi(X+X^{-1})$. 
Then we have
\[ h(1) - (-1)^{(\deg h)/2}h(-1)
= \phi(2) - \phi(-2) \equiv 0 \mod 4
\]
and thus $(-1)^{(\deg h)/2}h(1)h(-1) \equiv h(1)^2 \equiv 1$ mod $4$.
This means that $(-1)^{(\deg h)/2}h(1)h(-1) =\allowbreak 1$ or 
$-3$ in $\bQ_2^\times/\bQ_2^{\times2}$. 
Since $g$ is of type $2$, we obtain 
\[ (-1)^{(\deg f)/2}f(1)f(-1)
=(-1)^{(\deg g)/2}g(1)g(-1)\cdot(-1)^{(\deg h)/2}h(1)h(-1)
= 1 \text{ or $-3$}
 \]
in $\bQ_2^\times/\bQ_2^{\times2}$.
\end{proof}

Let $F\in\bZ[X]$ be a $*$-symmetric polynomial of even degree $2n$ with 
the condition \eqref{eq:Sqcd}, and assume that $n_+\neq 1$ and $n_- \neq 1$, 
where $n_\pm$ denotes the multiplicity of $X\mp 1$ in $F$. 
In addition, let $r,s\in \bZ_{\geq 0}$, $\iota\in \Idx_{r,s}(F)$ and 
$D_\pm$ be as in \S\ref{sec:LGP}.
For a subset $H = \{f_1, \ldots, f_l\}\subset I$, 
we may identify $H$ with the polynomial $\prod_{j = 1}^l f_j^{n_j}(X)$, 
where $n_j$ is the multiplicity of $f_j$ in $F$. 
In particular, we write 
$H(\pm1) = \prod_{j = 1}^l f_j^{n_j}(\pm 1)$ and 
$\deg H = \sum_{j = 1}^l n_j\deg f_j$.
Recall that we define the graph $G(F)$ in \S \ref{ss:OGandOM}.

\begin{proposition}\label{prop:ECC_Sq}
Let $F$ and $\iota$ be as above.
Each connected component of $G(F)$ has even degree and satisfies 
the condition \eqref{eq:Sqcd}.
\end{proposition}
\begin{proof}
We begin with the case $F(1)F(-1)\neq 0$.
Let $H$ be a connected component of $G(F)$. The degree of $H$ is even 
since $H$ has no type $0$ component. 
First, suppose that $|H(\pm1)|$ were not a square. 
Then there would exist a prime $p$ such that
$v_p(H(\pm1)) \equiv 1 \bmod 2$. Furthermore, there would exist $g\in I\setminus H$ 
satisfying $v_p(g(\pm1)) \equiv 1 \bmod 2$ because $v_p(F(\pm1))$ is even. 
Therefore, Lemma \ref{lem:nosymmfactor} shows that $H$ and $g$ have 
irreducible $*$-symmetric factors in $\bZ_p[X]$ which 
are divisible by $X\mp1$ in $\bF_p[X]$. This means that 
there exists $f\in H$ such that $p\in\Pi_{f,g}$. 
Hence the factor $g$ is connected to $H$ in $G(F)$, that is, $g\in H$. 
This is a contradiction, so $|H(1)|$ and $|H(-1)|$ are squares. 

Next, suppose that $(-1)^{(\deg H)/2} H(1)H(-1)$ were not a square. 
Since $|H(1)|$ and $|H(-1)|$ are squares, we would have 
$(-1)^{(\deg H)/2} H(1)H(-1) = -1$ mod squares. 
On the other hand, there exists $g\in I\setminus H$ such that 
$(-1)^{(\deg g)/2} g(1)g(-1) \notin\{1, -3\}$ in $\bQ_2^\times/\bQ_2^{\times2}$
because $(-1)^nF(1)F(-1)$ is a square (in $\bQ$, and thus in $\bQ_2$).
These mean that $H$ and $g$ have an irreducible $*$-symmetric factors in 
$\bZ_2[X]$ which are divisible by $X-1$ in $\bF_2[X]$, by Lemma \ref{lem:nosymmfactor}.
Thus $g$ is connected to $H$ in $G(F)$. 
This is a contradiction, so $(-1)^{(\deg H)/2} H(1)H(-1)$ is a square.
We have now proved the case $F(1)F(-1)\neq 0$. 

Let us proceed to the case $F(1)F(-1)= 0$. 
We denote by $G(F)'$ the graph obtained by removing $X-1$ and $X+1$ from $G(F)$. 
Let $K$ be a connected component of $G(F)'$, and 
$H$ the connected component of $G(F)$ containing $K$.
From the assumption that $n_+\neq 1$ and $n_-\neq 1$, it follows that $\deg H$ is even.  
If $K$ satisfies \eqref{eq:Sqcd} then so does $H$, and we are done.  
Let $|K(\pm1)|$ be not a square. Then $v_p(K(\pm1))\equiv 1 \bmod 2$ for some
prime $p$. Notice that there is no connected component $K'$ of $G(F)'$ 
satisfying $v_p(K'(\pm1))\equiv 1 \bmod 2$, 
since otherwise $K'$ would be connected to $K$. 
If $n_\pm \geq 3$ then we see easily that $X\mp1$ is connected to $K$ in $G(F)$ by 
Lemma \ref{lem:nosymmfactor}. This means that $X\mp1 \in H$, and hence 
$H(\pm1) = 0$ is a square. 
If $n_\pm = 2$ then 
\[v_p(D_\pm) = v_p(F_{12}(\pm1)) \equiv v_p(K(\pm1)) \equiv 1 \mod 2, \]
and in particular, $D_\pm \neq -1$ in $\bQ_p^\times/\bQ_p^{\times 2}$.
Therefore $X\mp1$ is connected to $K$ in $G(F)$, and  $H(\pm1) = 0$ is a square.
To summarize, if $|K(1)|$ or $|K(-1)|$ is not a square then 
$H$ satisfies the condition $\eqref{eq:Sqcd}$. 

The remaining case is where $|K(1)|$ and $|K(-1)|$ are squares but
$(-1)^{(\deg K)/2}\allowbreak K(1)K(-1) = -1$ mod squares.  
If $n_+\geq 3$ or $n_-\geq 3$ then Lemma \ref{lem:nosymmfactor} implies that
$2\in \Pi_{X-1,f}$ or $2\in \Pi_{X+1,f}$ for some $f\in K$, 
and thus $X-1\in H$ or $X+1\in H$. 
Hence $H$ satisfies the condition \eqref{eq:Sqcd}, and we are done. 
Let $(n_+,n_-) = (2,2), (2,0)$ or $(0, 2)$, and 
let $K^c$ be the complement of $K$ in $G(F)'$.
Lemma \ref{lem:nosymmfactor} implies that 
$(-1)^{(\deg K^c)/2}K^c(1)K^c(-1) = 1$ or $-3$ in $\bQ_2^\times/\bQ_2^{\times2}$, 
since otherwise $K^c$ would be connected to $K$. 
Hence 
\[\begin{split}
&(-1)^{(\deg F_{12})/2}F_{12}(1)F_{12}(-1) \\
&\quad = (-1)^{(\deg K)/2}K(1)K(-1)\cdot(-1)^{(\deg K^c)/2}K^c(1)K^c(-1) 
= -1 \text{ or $3$}  
\end{split} \]
in $\bQ_2^\times/\bQ_2^{\times2}$, 
which implies that 
\[\begin{split}
D_+ D_-
&= (-1)^{(n_+ - \iota(X-1))/2} (-1)^{(n_- - \iota(X+1))/2}\, |F_{12}(1)F_{12}(-1)| \\
&= (-1)^s F_{12}(1)F_{12}(-1)\\
&= (-1)^n F_{12}(1)F_{12}(-1)\\
&= (-1)^{(n_+ + n_-)/2} (-1)^{(\deg F_{12})/2} F_{12}(1)F_{12}(-1)\\
&= \begin{cases}
1  \text{ or $-3$} & \text{if $(n_+,n_-) = (2,0)$ or $(0,2)$}\\
-1  \text{ or $3$} & \text{if $(n_+,n_-) = (2,2)$}
\end{cases}
\end{split} \]
in $\bQ_2^\times/\bQ_2^{\times2}$.
In particular, in $\bQ_2^\times/\bQ_2^{\times2}$, 
we have $D_+\neq -1$ if $(n_+,n_-) = (2,0)$ and  
$D_-\neq -1$ if $(n_+,n_-) =\allowbreak (0,2)$, and $D_+\neq -1$ or $D_+\neq -1$
if $(n_+,n_-) = (2,2)$.
These imply that $X-1\in H$ or $X+1\in H$, 
and therefore, the component $H$ satisfies the condition \eqref{eq:Sqcd}.
\end{proof}

For a polynomial $f\in \bZ[X]$ with $f(1)f(-1)\neq 0$, 
let $e(f)\in\{1,-1\}$ denote the sign of $(-1)^{(\deg f)/2} \allowbreak f(1)f(-1)$.
We remark that if $f$ satisfies \eqref{eq:Sqcd} then $e(f) = 1$. 

\begin{lemma}\label{lem:prolong}
Let $F$ be as above. Then $F$ satisfies the condition \eqref{eq:Sgcd} for 
$(r,s) = (n,n)$. 
Moreover, a map $\iota_0$ from $I_0(\bR)$ to $\bZ$ satisfying 
the condition \eqref{eq:idxmap0} in \textup{\S \ref{ss:IPSoverR}} and the equation 
$\iota_0(X-1) + \iota_0(X+1) = e(F_{12})-1$ 
can be prolonged to an index map in $\Idx_{n,n}(F)$. 
Here, if $X\mp1\notin I_0(\bR)$ then we understand that $\iota_0(X\mp 1) = 0$. 
\end{lemma}
\begin{proof}
The inequality $n\geq m(F)$ is obvious. 
Let $\phi$ be the trace polynomial of $F_{12}$. 
We remark that there is a
$1:2$ correspondence between the roots of $\phi$ 
which are in the interval $(-2,2)$ and those of $F_{12}$ which lie on the unit circle.
Because the sign of $\phi(2)\phi(-2)$ is $e(F_{12})$, 
the number of roots of $F_{12}$ which lie on the unit circle is $1-e(F_{12})$ 
mod $4$. 
Hence, if $F(1)F(-1)\neq 0$ (i.e. $F=F_{12}$) then 
\[m(F) \equiv n - (1-e(F_{12})) \equiv n \mod 2. \]
This means that $F$ satisfies \eqref{eq:Sgcd} for $(n,n)$. 

Moreover, the fact that
the number of roots of $F_{12}$ which lie on the unit circle is $1-e(F_{12})$ mod $4$
implies that we can define a map $\iota_1:I_1(\bR)\to\bZ$ satisfying 
the condition \eqref{eq:idxmap1} in \S \ref{ss:IPSoverR}
and the equation $\sum_{f\in I_1(\bR)}\iota_1(f) = 1-e(F_{12})$. 
If $\iota_0:I_0(\bR)\to I$ is a map satisfying the condition \eqref{eq:idxmap0} 
and the equation $\iota_0(X-1) + \iota_0(X+1) = e(F_{12})-1$ 
then the sum $\iota_0\oplus\iota_1:I_0(\bR)\cup I_1(\bR)\to\bZ$, 
which is an extension of $\iota_0$, 
belongs to $\Idx_{n,n}(F)$ because this map satisfies equation \eqref{eq:idxmapsum} for 
$(r,s) = (n,n)$. 
\end{proof}

Note that the graph $G(F)$ is determined by the polynomial $F$ and values 
$\iota(X-1)$ and $\iota(X+1)$ where $\iota$ is an index map, 
even if $n_+ = 2$ or $n_- = 2$.
We now prove Theorem \ref{th:mainapp2}.
\smallskip
\\
\textit{Proof of Theorem \textup{\ref{th:mainapp2}}.}
It is sufficient to show that we can define an index map in $\Idx_{n,n}(F)$ such that 
the following holds for each connected component $H$ of $G(F)$:
\begin{equation}\label{eq:conn_comp}
\begin{split}
&\text{$H$ is realized as the characteristic polynomial of a semisimple isometry}\\
&\text{of an even unimodular lattice of signature $(d,d)$, where $d := (\deg H)/2$.}  
\end{split}
\tag{$*$}
\end{equation}
We begin with the case where $n_+ \neq 2$ and $n_-\neq 2$.
In this case, the graph $G(F)$ depends on $F$ only, so does the set of connected 
components of $G(F)$. Let $H$ be a connected component of $G(F)$, and
put $d = (\deg H)/2$. 
Proposition \ref{prop:ECC_Sq} implies that $H$ satisfies the condition \eqref{eq:Sqcd}, 
and Lemma \ref{lem:prolong} implies that $H$ satisfies the condition \eqref{eq:Sgcd} 
for $(d,d)$. The obstruction group of $H$ is trivial for any index map since $H$ is 
connected, and then we have \eqref{eq:conn_comp} by 
Theorem \ref{th:main1}. Therefore, if we choose $\iota_H\in \Idx_{d,d}(H)$ for each
connected component $H$, then the sum $\oplus_H\, \iota_H \in \Idx_{n,n}(F)$ 
is the required index map. 

We proceed to the case $n_+ = 2$ or $n_- = 2$. 
We deal with the case $n_+ = 2$ only because the case $n_- = 2$ is similar.
For $\iota_0:I_0(\bR)\to \bZ$ mentioned in Lemma \ref{lem:prolong}, 
let $G(F;\iota_0)$ denote the graph determined from $(F,\iota)$ for an  
extension $\iota\in \Idx_{n,n}(F)$ of $\iota_0$, which is independent of 
the choice of $\iota$. 
If we can define $\iota_0:I_0(\bR)\to \bZ$ mentioned in Lemma \ref{lem:prolong}
so that \eqref{eq:conn_comp} holds for the connected components 
containing $X-1$ and $X+1$ of $G(F;\iota_0)$ respectively, 
then, as above, \eqref{eq:conn_comp} holds for 
the others and we obtain the required index map. 
Therefore, it suffices to define such a map $I_0(\bR)\to \bZ$. 
We remark that the connected component containing $X-1$ and that containing 
$X-1$ may coincide. 

First, let $n_+ = 2$ and $n_- > 2$. 
If $|F_{12}(1)| = 1$ in $\bQ_2^\times/\bQ_2^{\times2}$, 
we set $\iota_0(X-1) = -2$ and $\iota_0(X+1) = e(F_{12})+1$.
Then $D_+ = 1 \neq -1$ in $\bQ_2^\times/\bQ_2^{\times2}$, and in particular 
$2\in \Pi_{X-1,X+1}$. Thus $X-1$ and $X+1$ are in a common connected component 
$H$ of $G(F;\iota_0)$. 
Because any connected component $H'$ other than $H$ satisfies \eqref{eq:Sqcd} 
by Proposition \ref{prop:ECC_Sq} and in particular $e(H') = 1$, 
we have $e(H^\circ) = e(F_{12})$ where $H^\circ = H\setminus\{X-1,X+1\}$. 
Therefore $\iota_0$ can be prolonged to a map in $\Idx_{d,d}(H)$ 
by Lemma \ref{lem:prolong}, where $d = (\deg H)/2$. Then \eqref{eq:conn_comp} 
holds for $H$ since the obstruction group of $H$ is trivial.
If $|F_{12}(1)| \neq 1$ in $\bQ_2^\times/\bQ_2^{\times2}$, 
we set $\iota_0(X-1) = 0$ and $\iota_0(X+1) = e(F_{12})-1$.
Then, as above, the set $\Pi_{X-1, X+1}$ has the prime $2$, and 
we obtain \eqref{eq:conn_comp} for the connected component containing $X-1$ and $X+1$. 
Thus we are done. 

Next, let $n_+ = 2$ and $n_- = 0$. 
In this case we define $\iota_0(X-1) = e(F_{12})-1$.
If $H$ is the connected component containing $X-1$, 
then we get $e(H\setminus\{X-1\}) = e(F_{12})$ as above. Moreover $\iota_0$ can be 
prolonged to a map 
in $\Idx_{d,d}(H)$, and the condition \eqref{eq:conn_comp} holds for $H$ as above. 
Therefore we are done. 

Finally, let $n_+ = 2$ and $n_- = 2$. 
If $|F_{12}(-1)|$ is neither $1$ nor $-1$ in $\bQ_2^\times/\bQ_2^{\times2}$, we set  
\begin{alignat*}{2}
&\text{$\iota_0(X-1) = -2$ and $\iota_0(X+1) = e(F_{12})+1$} \quad
&\text{if $|F_{12}(1)| = 1$ in $\bQ_2^\times/\bQ_2^{\times2}$,}\\
&\text{$\iota_0(X-1) = 0$ and $\iota_0(X+1) = e(F_{12})-1$} 
&\text{if $|F_{12}(1)| \neq 1$ in $\bQ_2^\times/\bQ_2^{\times2}$.}    
\end{alignat*}
Then we can get $D_+\neq -1$ and $D_-\neq -1$ in $\bQ_2^\times/\bQ_2^{\times2}$
and $2\in \Pi_{X-1, X+1}$.  
This implies as above that \eqref{eq:conn_comp} holds for the connected component 
containing $X-1$ and $X+1$. 
We can also finish the proof similarly in the case where $|F_{12}(1)| \neq 1, -1$ 
in $\bQ_2^\times/\bQ_2^{\times2}$.

Suppose that $|F_{12}(1)| = 1$ and $|F_{12}(-1)| = 1$ in $\bQ_2^\times/\bQ_2^{\times2}$. 
We only deal with this case, because one can finish the proof similarly
in the remaining cases, i.e., the cases 
$(|F_{12}(1)|, |F_{12}(-1)|) =\allowbreak (1,-1), (-1,1)$ and $(-1,-1)$ 
in $\bQ_2^\times/\bQ_2^{\times2}$.
If $e(F_{12}) = 1$, we set $\iota_0(X-1) = 2$ and $\iota_0(X+1) = -2$. 
Then, as above, we have $2\in \Pi_{X-1, X+1}$, and \eqref{eq:conn_comp} holds 
for the connected component containing $X-1$ and $X+1$. 
Let $e(F_{12}) = -1$, and let $G(F)'$ denote the graph obtained by removing 
$X-1$ and $X+1$ from $G(F)$. We remark that 
$G(F)$ may depends on the choice of $\iota_0$ but $G(F)'$ is independent of it. 
As discussed in the proof of Proposition \ref{prop:ECC_Sq}, 
if $K$ is a connected component of $G(F)'$ such that $|K(\pm1)|$ is not a square
then $X\mp1$ is connected to $K$ in $G(F)$ independently of the choice of $\iota_0$. 
Let $K_\pm$ denote the union of connected components $K\subset I_1$ of $G(F)'$ 
such that $|K(\pm1)|$ is not a square. 
If $K_+\cap K_-\neq \emptyset$ then there exists a connected component $K$ 
which contains an element of $K_+\cap K_-$. 
In this case, $X-1$ and $X+1$ are connected via $K$
and contained in a common connected component $H$ of $G(F)$
(independently of the choice of $\iota_0$).  
Hence we obtain \eqref{eq:conn_comp} for $H$ as required. 
Assume that $K_+\cap K_- = \emptyset$, and set $K_0 := I_1\setminus (K_+\cup K_-)$.
There are the following four cases:
\begin{itemize}
\item[(a)] $e(K_+)=1, e(K_0)=1$ and $e(K_-)=-1$.
\item[(b)] $e(K_+)=1, e(K_0)=-1$ and $e(K_-)=1$.
\item[(c)] $e(K_+)=-1, e(K_0)=1$ and $e(K_-)=1$.
\item[(d)] $e(K_+)=-1, e(K_0)=-1$ and $e(K_-)=-1$.
\end{itemize}
Notice that if $e(K_0)= -1$, that is, in the cases (b) and (d), 
then it follows from Lemma \ref{lem:nosymmfactor} that each connected component $K$ of 
$K_0$ has an irreducible $*$-symmetric factor in $\bZ_2[X]$ which is divisible by $X-1$ 
in $\bF_2[X]$.  
Moreover, if we set $\iota_0(X+1) = -2$ then $X+1$ and $K$ are connected in 
$G(F;\iota_0)$ since $D_- = 1 \neq -1$ in $\bQ_2^\times/\bQ_2^{\times2}$. 

In the case (a),(b) or (c), we set
\[\iota_0(X-1) = \begin{cases}
  0 & \text{in the case (a) or (b)}\\ 
  -2  & \text{in the case (c)}
\end{cases}
\text{\quad and \quad}
\iota_0(X+1) = \begin{cases}
  -2 & \text{in the case (a) or (b)}\\ 
  0  & \text{in the case (c)}.
\end{cases}\]
If $H_+$ and $H_-$ denote the connected components of $G(F;\iota_0)$ 
containing $X-1$ and $X+1$ respectively ($H_+$ and $H_-$ may coincide), 
then Lemma \ref{lem:prolong} shows that 
$\iota_0$ can be prolonged to index maps in $\Idx_{d_+, d_+}(H_+)$ and in 
$\Idx_{d_-, d_-}(H_-)$ respectively, where $d_\pm = \deg H_\pm$.  
Thorefore, we can obtain \eqref{eq:conn_comp} for $H_+$ and $H_-$, and we are done. 

Let us proceed to the case (d), and let $\phi$ be the trace polynomial of $K_+$. 
We remark that $\phi(2)$ and $\phi(-2)$ have different signs since $e(K_+) = -1$, 
and that 
\begin{align*}
&|\phi(2)| = |K_+(1)| = |F_{12}(1)| = 1 
\quad\text{in  $\bQ_2^\times/\bQ_2^{\times2}$ and} \\
&\text{$|\phi(-2)| = |K_+(-1)|$ is a square (in $\bQ^\times$).}
\end{align*}
These imply that
\[ (-1)^{(\deg K_+)/2}K_+(1)K_+(-1)
= \phi(2)\phi(-2)
= - |\phi(2)\phi(-2)|
= -1 \quad\text{in $\bQ_2^\times/\bQ_2^{\times2}$.}
\]
%
%
Then, Lemma \ref{lem:nosymmfactor} shows that there are connected components 
$K$ of $K_+$ and $K'$ of $K_0$ such that $K$ and $K'$ are connected in $G(F)'$. 
Now, we define $\iota_0:I_0(\bR)\to\bZ$ by 
$\iota_0(X-1) = 0$ and $\iota_0(X+1) = -2$.
Then $X-1, K, K'$ and $X+1$ are connected and contained in a common connected 
component of $G(F;\iota_0)$, and we obtain \eqref{eq:conn_comp} for this component.
This completes the proof. 
\qed\medskip

The above proof is not valid if $n_+ = 1$ or $n_- = 1$, but 
Theorem \ref{th:mainapp2} seems to hold even if $n_+ = 1$ or $n_- = 1$.

\section{Automorphisms of K3 surfaces}\label{sec:Aut_of_K3}
A \textit{K3 surface} is a simply connected  compact complex surface with
a nowhere vanishing holomorphic $2$-form.
We prove Theorem \ref{th:mainK3} in this section. 

\subsection{Lifting}
Let $\cX$ be a K3 surface with a nowhere vanishing holomorphic $2$-form $\omega_\cX$. 
The middle cohomology group $H^2(\cX,\bZ)$ has the intersection form 
$\la\cdot, \cdot \ra$, which makes 
$H^2(\cX,\bZ)$ an even unimodular lattice of signature $(3, 19)$. 
Such a lattice is called a \textit{K3 lattice}, which is uniquely determined up to
isomorphism.
The form $\la\cdot, \cdot \ra$ is extended on $H^2(\cX, \bC)$ as a hermitian form.
The \textit{Hodge structure} is the direct sum decomposition 
\[ H^2(\cX,\bC) = H^{2,0}(\cX)\oplus H^{1,1}(\cX) \oplus H^{0,2}(\cX) \]
where $H^{2,0}(\cX) = \bC\omega_\cX, H^{0,2}(\cX) = \bC\overline{\omega_\cX}$
and $H^{1,1}(\cX) = (H^{2,0}(\cX)\oplus H^{0,2}(\cX))^\perp$.
Here $\overline{\,\cdot\,}$ is the complex conjugate.
The real part $H^{1,1}_\bR(\cX) = \{ x\in H^{1,1}(\cX) \mid \overline{x} = x \}$
of $H^{1,1}(\cX)$ is of signature $(1, 19)$, and thus 
$\cC_\cX := \{ x\in H^{1,1}_\bR(\cX) \mid \la x,x\ra > 0 \}$ has exactly two connected 
components. The one containing a K\"{a}hler class is called the 
\textit{positive cone} and denoted by $\cC_\cX^+$.
The intersection $P_\cX := H^2(\cX,\bZ)\cap H^{1,1}(\cX)$ is called the 
\textit{Picard lattice} of $\cX$.
Set $\Delta_\cX := \{ x\in P_\cX \mid \la x,x \ra = -2 \}$
and $\Delta_\cX^+ := \{ x\in \Delta_\cX \mid \text{$x$ is effective} \}$.
Then the \textit{K\"{a}hler cone} $\cK_\cX$ is written as
\[\cK_\cX = \{ x\in \cC_\cX^+ \mid \la x,r \ra>0 \quad\text{for all $r\in \Delta_\cX^+$} \}, \]
see \cite[Chapter 8, Theorem 5.2]{Hu16}.
Notice that the K\"{a}hler cone $\cK_\cX$ is a Weyl chamber with respect to 
the root system $\Delta_\cX$.

Conversely, we can define a Hodge structure and a K\"{a}hler cone on a K3 lattice 
formally.
Let $(\Lambda, b)$ be a K3 lattice, and set $\Lambda_\bC = \Lambda\otimes \bC$.
A vector $\omega\in \Lambda_\bC$ such that the signature of 
$\bC\omega \oplus \bC\overline\omega \subset \Lambda_\bC$ is $(2, 0)$ 
gives the decomposition
\begin{equation}\label{eq:hodgestr}
\Lambda_\bC = H^{2,0}\oplus H^{1,1}\oplus H^{0,2}
\quad \text{where } H^{2,0} = \bC\omega,  
H^{1,1}=\{ \omega, \overline{\omega} \}^\perp,
H^{0,2} = \bC\overline{\omega}. 
\end{equation}
Such a decomposition or such a vector $\omega$ is called a \textit{Hodge structure} 
on $\Lambda$. 
Fix a Hodge structure $\omega\in \Lambda_\bC$ and one connected component $\cC^+$ of 
$\cC := \{ x\in H^{1,1}_\bR \mid b(x,x) > 0 \}$ in the real part $H^{1,1}_\bR$
of $H^{1,1}$. We define the \textit{Picard lattice} $P$ of $\Lambda$ to be 
$\Lambda \cap H^{1,1}$. Then the set $\Delta := \{ x\in P \mid b(x,x) = -2 \}$ 
is a root system. A \textit{K\"{a}hler cone} of $\Lambda$ is a 
Weyl chamber in $\cC^+$.
If $\cK$ is a K\"{a}hler cone of $\Lambda$ then the pair $(\omega, \cK)$
is called a \textit{K3 structure} of $\Lambda$.

The following theorem is a well-known consequence of the Torelli theorem and
surjectivity of period mapping. See \cite[\S 6]{Mc11} for instance. 
\begin{theorem}\label{th:lift}
Let $\Lambda$ be a K3 lattice equipped with a K3 structure 
$(\omega, \cK)$, and $t$ an isometry on $\Lambda$ preserving the K3 structure.
Then there exists a K3 surface $\cX$, an automorphism $\varphi$ on $\cX$, and a 
lattice isometry $\tau:H^2(\cX,\bZ) \to \Lambda$ such that
$\tau(\omega_\cX) \in \bC\omega$,
$\tau(\cK_\cX) = \cK$, 
and the diagram
\[ \xymatrix{
H^2(\cX,\bZ) \ar[r]^{\varphi^*} \ar[d]^\tau & H^2(\cX,\bZ) \ar[d]^\tau \\
\Lambda \ar[r]^t & \Lambda
} \]
commutes.
\end{theorem}

\subsection{Entropy}
Let $\cX$ be a K3 surface, and $\varphi$ an automorphism on $\cX$. 
It is known that the topological entropy of $\varphi$ is given by
$\log\lambda(\varphi^*)$, where $\lambda(\varphi^*)$ is the spectral 
radius of $\varphi^* : H^2(\cX)\to H^2(\cX)$, that is, the maximum absolute value of 
the eigenvalues of $\varphi^*$.
On the other hand, the characteristic polynomial of $\varphi^*$ is either
\begin{itemize}
\item a product of cyclotomic polynomials, or
\item a product of one Salem polynomial and a product of cyclotomic polynomials.
\end{itemize}
Hence the entropy of $\varphi$ is $0$ or the logarithm of a Salem number, 
see \cite[Section 3]{Mc02}.

A polynomial $F(X)$ of degree $22$ is called a \textit{complemented Salem polynomial}
if $F(X)$ can be expressed as $F(X)=S(X)C(X)$, where $S(X)$ is a Salem polynomial and 
$C(X)$ is a product of cyclotomic polynomials.  
In this case, $S$ is called the Salem factor of $F$. 
The facts mentioned above imply that
if a K3 surface automorphism $\varphi$ has positive entropy, then 
the characteristic polynomial of $\varphi^*$ is a 
complemented Salem polynomial.

\begin{definition}
A Salem number $\lambda$ is \textit{projectively} (resp. \textit{nonprojectively}) 
\textit{realizable} if there exists a projective (resp. nonprojective) K3 
surface and an automorphism on it of topological entropy $\log \lambda$. 
\end{definition}

\begin{notation}
Let $F$ be a complemented Salem polynomial with Salem factor $S$.
For each root $\delta$ of $S$ with $|\delta| = 1$, 
the symbol $\iota_\delta$ denotes the index map in $\Idx_{3,19}(F)$ defined by
\[ \iota_\delta(f) = \begin{cases}
  2 & \text{if $f(X) = X - (\delta + \delta^{-1})X + 1$}\\
  -2n_f & \text{if $f(X) = X - (\zeta + \zeta^{-1})X + 1$ for some 
  $\zeta\in \bT\setminus\{\pm1, \delta^{\pm 1} \}$} \\
  -n_f & \text{if $f(X) = X\mp 1$}, 
\end{cases}
\]
where $n_f$ is the multiplicity of $f$ in $F$ and 
$\bT = \{ \delta\in\bC \mid |\delta| = 1 \}$.
\end{notation}

The nonprojective case is tractable thanks to the following proposition.

\begin{proposition}\label{prop:nonproj_realizable}
Let $\lambda$ be a Salem number with $4\leq \deg\lambda  \leq 22$, and 
$S$ its minimal polynomial.
The following are equivalent:
\begin{enumerate}
\item $\lambda$ is nonprojectively realizable.
\item There exists a conjugate $\delta$ of $\lambda$ with $|\delta| = 1$ and 
a complemented Salem polynomial $F$ with Salem factor $S$ such that
a K3 lattice has a semisimple $(F, \iota_\delta)$-isometry. 
\end{enumerate}
\end{proposition}
\begin{proof}
\textup{(i) $\Rightarrow$ (ii).}
Let $\varphi$ be an automorphism on a nonprojective K3 surface $\cX$ with
entropy $\log\lambda$. 
Then $\varphi^*:H^2(\cX, \bZ)\to H^2(\cX, \bZ)$ is a semisimple isometry
on a K3 lattice with characteristic polynomial $F$, where $F$ is a complemented 
Salem polynomial with Salem factor $S$.  
Let $\delta\in \bT$ be the root of $F$ such that 
$\varphi^*\omega_\cX = \delta \omega_\cX$. Set
$\Lambda(S;\varphi^*) := \{ x\in H^2(\cX, \bZ) \mid S(\varphi^*)x = 0 \}$.
If $\delta$ were a root of unity, then 
$\Lambda(S;\varphi^*)$ would be orthogonal to $H^{2,0}(\cX)\oplus H^{0,2}(\cX)$, and 
\[ \Lambda(S;\varphi^*)\subset H^{1,1}(\cX)\cap H^2(\cX, \bZ) = P_\cX  \]
would hold. In particular, the Picard lattice $P_\cX$ contains
an element whose self intersection number is positive. 
This contradicts the nonprojectivity of $\cX$.
Therefore $\delta\in \bT$ must be a root of $S$, and 
this means that $\idx_{\varphi^*} = \iota_\delta$. 

\textup{(ii) $\Rightarrow$ (i).}
Let $\delta$ be a conjugate of $\lambda$ with $|\delta| = 1$, and let 
$F(X) = S(X)C(X)$ be a complemented Salem polynomial with Salem factor $S$. Suppose that
a K3 lattice $\Lambda$ has a semisimple $(F, \iota_\delta)$-isometry $t$. 
Then, an eigenvector $\omega\in \Lambda_\bC$ of $t$ corresponding to $\delta$ 
define the Hodge structure \eqref{eq:hodgestr} preserved by $t$, since
$\iota_\delta(X^2 -(\delta + \delta^{-1})X + 1) = 2$.
In this case, the Picard lattice is written as
$P = \{x\in \Lambda\mid C(t)x = 0 \}$ by \cite[Theorem 7.4]{IT22}, 
and in particular, $P$ is negative definite. 
Let us fix a K\"{a}hler cone $\cK$. 
In general, the isometry $t$ maps $\cK$ to another chamber.
However, the Weyl group, that is, the subgroup of $\rO(P)$ generated by all root 
reflections, has a unique element $w$ such that 
$w(t(\cK)) = \cK$, because the Weyl group acts on the set of Weyl chambers 
simply transitively. Here $w$ is extended to an isometry on $\Lambda$ 
by letting $w$ act on $P^\perp$ as the identity. 
Then, the composition $w\circ t$ preserves the 
K3 structure $(\omega, \cK)$ and has spectral radius $\lambda$. 
Theorem \ref{th:lift} means that there exists a K3 surface $\cX$ and 
an automorphism $\varphi$ on $\cX$ with entropy $\log \lambda$.
Furthermore, the K3 surface $\cX$ is nonprojective because its Picard lattice is 
negative definite.   
\end{proof}

We finally prove Theorem \ref{th:mainK3}. 
\smallskip
\\
\textit{Proof of Theorem \textup{\ref{th:mainK3}}.}
Let $\lambda$ be a Salem number of degree $20$, and $S$ its minimal
polynomial. Fix a conjugate $\delta$ of $\lambda$ with $|\delta| = 1$. 
Then, a K3 lattice has a semisimple $((X-1)(X+1)S(X), \iota_\delta)$-isometry 
by Theorem \ref{th:mainapp}.
This implies that $\lambda$ is nonprojectively realizable by 
Proposition \ref{prop:nonproj_realizable}. \qed
\vspace{5mm}\noindent
\textbf{Acknowledgments}.
This work is supported by JSPS KAKENHI Grant Number JP21J20107. 
The author thanks Eva Bayer-Fluckiger for her kind replies to 
his questions and for invaluable coments on earlier versions of the manuscript. 
He also thanks Katsunori Iwasaki for giving a lot of helpful advice during the 
preparation of this paper. 
In particular, the explanation in \S\ref{ss:OGandOM} became clearer thanks to 
discussions with him.


\begin{thebibliography}{99}


\bibitem[BT20]{BT20}
E.~Bayer-Fluckiger and L.~Taelman, 
\textit{Automorphisms of even unimodular lattices and equivariant Witt groups.}
J. Eur. Math. Soc. \textbf{22} (2020), 3467--3490.

\bibitem[Ba15]{Ba15}
E.~Bayer-Fluckiger, 
\textit{Isometries of quadratic spaces.}
J. Eur. Math. Soc. (JEMS) \textbf{17} (2015), no. 7, 1629--1656.

\bibitem[Ba20]{Ba20}
E.~Bayer-Fluckiger, 
\textit{Isometries of lattices and Hasse principles.}
arXiv:2001.07094v3.

\bibitem[Ba21]{Ba21}
E.~Bayer-Fluckiger, 
\textit{Isometries of lattices and automorphisms of K3 surfaces.}
arXiv:2107.07583v1.

\bibitem[Ba22]{Ba22}
E.~Bayer-Fluckiger, 
\textit{Automorphisms of K3 surfaces, signatures, and isometries of lattices.}
arXiv:2209.06698v2.

\bibitem[Br20]{Br20}
S.~Brandhorst, 
\textit{On the stable dynamical spectrum of complex surfaces.} 
Math. Ann. \textbf{377} (2020), no. 1-2, 421--434.

\bibitem[BCM03]{BCM03}
R.~Brusamarello, P.~Chuard-Koulmann and J.~Morales,
\textit{Orthogonal groups containing a given maximal torus.}
J. Algebra \textbf{266} (2003), no. 1, 87--101.


\bibitem[GM02]{GM02}
B.H.~Gross and C.T.~McMullen,
\textit{Automorphisms of even unimodular lattices and unramified Salem numbers.}
J. Algebra \textbf{257} (2002), no. 2, 265--290.




\bibitem[Hu16]{Hu16}
D.~Huybrechts,
{Lecture on K3 Surfaces.}
Cambridge University Press, Cambridge, 2016.


\bibitem[IT22]{IT22}
K.~Iwasaki and Y.~Takada,
\textit{Hypergeometric groups and dynamics on K3 surfaces.} 
Math. Z. \textbf{301} (2022), no. 1, 835--891. 

\bibitem[IT23]{IT23}
K.~Iwasaki and Y.~Takada,
\textit{K3 surfaces, Picard numbers and Siegel disks.} 
J. Pure Appl. Algebra \textbf{227} (2023), no. 3, Paper No. 107215.


\bibitem[Mc02]{Mc02}
C.T.~McMullen,
\textit{Dynamics on K3 surfaces: Salem numbers and Siegel disks.}
J. Reine Angew. Math. \textbf{545} (2002), 201--233.

\bibitem[Mc11]{Mc11}
C.T.~McMullen,
\textit{K3 surfaces, entropy and glue.}
J. Reine Angew. Math. \textbf{658} (2011), 1--25.

\bibitem[Mc16]{Mc16}
C.T.~McMullen,  
\textit{Automorphisms of projective K3 surfaces with minimum entropy.}
Invent. Math. \textbf{203} (2016), no. 1, 179--215.

\bibitem[Mi69]{Mi69}
J.~Milnor,
{On isometries of inner product spaces.}
Invent. Math. \textbf{8} (1969), 83--97.




\bibitem[Og10]{Og10}
K.~Oguiso,
\textit{The third smallest Salem number in automorphisms of K3 surfaces.} 
Algebraic geometry in East Asia-Seoul 2008, 331--360, 
Adv. Stud. Pure Math., \textbf{60}, Math. Soc. Japan, Tokyo, 2010. 


\bibitem[OM73]{OM73}
O.T.~O'Meara,
\textrm{Introduction to Quadratic Forms. Reprint of the 1973 edition.} 
Classics in Mathematics. Springer-Verlag, Berlin, 2000.

\bibitem[Sa63]{Sa63}
R.~Salem, 
\textrm{Algebraic Numbers and Fourier Analysis.}
D. C. Heath and Co., Boston, MA, 1963.


\bibitem[Se73]{Se73}
J.P.~Serre,
{A Course in Arithmetic.}
Springer Sience+Business Media New York, 1973.


\bibitem[Sc85]{Sc85}
W.~Scharlau, 
{Quadratic and Hermitian Forms.}
Springer-Verlag, Berlin, 1985.

\bibitem[Za62]{Za62}
H.~Zassenhaus,
\textit{On the spinor norm.}
Arch. Math. \textbf{13} (1962), 434--451.
\end{thebibliography}
\end{document}